\documentclass[11pt,amstex,amssymb]{amsart}
\usepackage{graphicx}
\usepackage{amscd}
\usepackage{amsmath}
\usepackage{amsfonts}
\usepackage{amssymb}
\providecommand{\U}[1]{\protect\rule{.1in}{.1in}}
\providecommand{\U}[1]{\protect\rule{.1in}{.1in}}
\providecommand{\U}[1]{\protect\rule{.1in}{.1in}}
\providecommand{\U}[1]{\protect\rule{.1in}{.1in}}

\newtheorem{theorem}{Theorem}[section]
\theoremstyle{plain}

\newtheorem{corollary}{Corollary}[section]
\newtheorem{definition}{Definition}[section]

\newtheorem{lemma}{Lemma}[section]

\newtheorem{proposition}{Proposition}[section]
\newtheorem{remark}{Remark}[section]
\numberwithin{equation}{section}

\begin{document}
\title[Holomorphic foliations with quasi-homogeneous separatrices]{A comprehensive approach to the moduli space of quasi-homogeneous singularities}
\author{Leonardo M. C\^{a}mara}
\address{Departamento de Matem\'{a}tica, CCE-UFES, Av. Fernando Ferrari 514, Campus de
Goiabeiras, Vit\'{o}ria-ES, 29075-910, Brazil.}
\email{leonardo.camara@ufes.br}
\author{Bruno Sc\'{a}rdua}
\address{Intituto de Matem\'{a}tica, Universidade Federal do Rio de Janeiro}
\email{scardua@im.ufrj.br }
\thanks{This work was partially supported by CAPES-PROCAD grant n$^{\text{\d{o}}}$ 0007056}
\thanks{2010 Math. Subj. Class. 37F75, 32S65, 32S45}

\begin{abstract}
We study the relationship between singular holomorphic foliations in
$(\mathbb{C}^{2},0)$ and their separatrices. Under mild conditions we describe
a complete set of analytic invariants characterizing foliations with
quasi-homogeneous separatrices. Further, we give the full moduli space of
quasi-homogeneous plane curves. This paper has  an expository character in order to make it accessible also to non-specialists. 

\end{abstract}
\maketitle

\section{Introduction}

In this this paper we deal with the classification of  germs of curves and
germs of holomorphic foliations in $(\mathbb{C}^{2},0)$ (cf. Theorems A and B). The problem of the classification of germs
of analytic plane curves has been addressed by several authors since the
XVII$^{\text{th}}$ century with different methods (see for instance \cite{Be
06}, \cite{Br 66}, \cite{HefHer 2007}). In the first part of the present work,
we study the problem of the analytic classification of germs of singular
curves with many branches from the viewpoint of Holomorphic Foliations. This
allows the use of geometrical techniques including the blow-up and holonomy
which are related to the study of normal forms for quasi-homogeneous
polynomials in two variables.

Next, we use the standard resolution of theses singularities in order to
stratify them and thus identify the moduli space of each stratum. As a
consequence, our method provides an effective way to identify if two
quasi-homogeneous curves are equivalent. Further, remark that the analytic
type of a quasi-homogeneous curve is one of the invariants which determine the
analytic type of a foliation having such a curve as separatrix set (cf. Theorem ~B). Therefore, the present classification completes the
classification of such germs of complex analytic foliations.

On the other hand, the problem of local classification of differential
equations of the form $Adx+Bdy=0$ in\ two variables has been studied by
various mathematicians --- since the end of the nineteenth century --- as C. A.
Briot, J. C. Bouquet, H. Dulac, H. Poincar\'{e}, I. Bendixson, G. D. Birkhoff,
C. L. Siegel, A. D. Brjuno \textit{et Al}. In the middle 1970s R. Thom
restored the interest in this question with a series of talks at IHES. In
fact, he conjectured that a germ of a foliation $\mathcal{F}$ in $(\mathbb{C}%
^{2},0)$ with a finite number of separatrices, i.e. a finite number of
analytic invariant curves through the origin, has its analytic type
characterized by its holonomy with respect to the separatrix set (cf.
\cite{CeMo 88}, pp. 162, 163). In \cite{MaMo 80}, \cite{MaRa 82}, and
\cite{MaRa 83} it is proved that the conjecture has an affirmative answer if
the linear part of the vector field defining the foliation is non-nilpotent.
In \cite{Mo 85} it is proved that the conjecture is not true in general with
the introduction of an analytic invariant called vanishing holonomy. Further,
in \cite{CaSa 82a} it is proved that any germ of a singular holomorphic
foliation in $(\mathbb{C}^{2},0)$ has a nonempty  separatrix set, which is denoted by
$\operatorname*{Sep}(\mathcal{F})$. Since this time, the problem of finding a
complete set of analytic invariants determining the analytic type of a germ of a foliation in $(\mathbb{C}^{2},0)$ having a finite number of separatrices is
known as Thom's problem (cf.\cite{EliIl'yShVo 93}, pp. 60, 98). In \cite{CeMo
88} the results of \cite{Mo 85} are generalized, classifying a Zariski open
subset of the nilpotent singularities in terms of the vanishing holonomy (now
called projective holonomy). Other contributions have been given by many
authors such as \cite{BeMzSa 99}, \cite{EliIl'yShVo 93}, \cite{Str 2003}, etc.

In \cite{Ma2000} the problem of moduli space is studied from the deformation
viewpoint. There it is proved that the moduli space of local unfoldings of
quasi-homogeous foliations is determined by the conjugacy class of the
projective holonomy and the analytic type of its separatrix set for a generic
class of foliations called quasi-hyperbolic (cf. \cite{Ma2000}, Definition
1.1, p. 255; Theorem B, p. 256; and Definition 6.8, p. 273). Namely, a germ of a foliation $\mathcal{F}$ is called \textit{quasi-hyperbolic generic} provided
that the following conditions are satisfyed: (i) its resolution $\widetilde
{\mathcal{F}}$ has at least one non-solvable projective holonomy; (ii)
$\widetilde{\mathcal{F}}$ has no saddle-nodes and the ratio between the
eigenvalues of each of its singular points is not a negative real number.
After, in \cite{Gz 2008} it is proved that any two quasi-hyperbolic generic
quasi-homogeous foliations can be linked by such kind of unfoldings,
classifying the quasi-hyperbolic generic quasi-homogeous foliations.

Here, from a quite different viewpoint, we show in the second part of this
work an analogous result with less restrictive hypotheses on the foliation $\mathcal{F}$
(cf. Theorem ~B), using a geometric and much simpler proof. In
fact, this geometrical approach leads also to the classification of curves.

Finally, we would like to remark that one of the main sources of inspiration
for this work was the relationship between singular holonomies (cf. e.g.
\cite{CaSc95}, \cite{CaSc99}, \cite{CaSc2000}, \cite{CaSc 2001}) and the
analytic type of a foliation near their Hopf components (see definition
below). Furthermore, our approach can be used to understand the moduli space
of more general germs of singular foliations, for instance, in the presence of
saddle nodes.

The plan of the  article is as follows. First we determine normal forms for
quasi-homogeneous algebraic curves obtaining some geometric properties for the
resolution of the separatrix set. With this geometric features at hand, we
determine the moduli space in terms of the moduli space of punctured Riemann
spheres. In the sequel, we study the semilocal invariants of resolved
foliation determining the analytic type of each Hopf component of the
foliation. Then we introduce natural cocycles that measure the obstruction for
two analytically componentwise equivalent foliations to be really analytically
equivalent. Finally we use the geometric description of the separatrix set in
order to trivialize these cocycles and construct an explicit conjugation
between two foliations with the same quasi-homogeneous curve and analytically
conjugate projective holonomies.

\part{Classification of curves}

\section{Preliminaries}

Let $C$ be a singular curve and  $\pi:(\mathcal{M},D)\longrightarrow
(\mathbb{C}^{2},0)$ its standard resolution, i.e. the minimal resolution of
$C$ whose \textit{strict transform} $\widetilde{C}:=\pi^{-1}(C)\backslash D$
\ is transversal to the exceptional divisor $D=\pi^{-1}(0)$. A germ of a holomorphic function $f\in\mathbb{C}\{x,y\}$ is said to be
\textit{quasi-homogeneous} if there is a local system of coordinates in which
$f$ can be represented by a quasi-homogeneous polynomial, i.e. $f(x,y)=\sum
_{ai+bj=d}a_{ij}x^{i}y^{j}$ where $a,b,d\in\mathbb{N}$. Let $M$ be a manifold
and $M_{\Delta}(n):=\{(x_{1},\cdots,x_{n})\in M^{n}:x_{i}\neq x_{j}$ for all
$i\neq j\}$. Let $S_{n}$ denote the group of permutations of $n$ elements and
consider its action in $M_{\Delta}(n)$ given by $(\sigma,\lambda)\mapsto
\sigma\cdot\lambda=(\lambda_{\sigma(1)},\cdots,\lambda_{\sigma(n)})$. The
quotient space induced by this action is denoted by $\operatorname*{Symm}%
(M_{\Delta}(n))$. Now suppose a Lie group $G$ acts in $M$ and let $G$ act in
$M_{\Delta}(n)$ in the natural way $(g,\lambda)=(g\cdot\lambda_{1}%
,\cdots,g\cdot\lambda_{n})$ for every $\lambda\in M_{\Delta}(n)$. Then the
actions of $G$ and $S_{n}$ in $M_{\Delta}(n)$ commute. Thus one obtains a
natural action of $G$ in $\operatorname*{Symm}(M_{\Delta}(n))$. Given
$\lambda\in M_{\Delta}(n)$, denote its equivalence class in
$\operatorname*{Symm}(M_{\Delta}(n))/G$ by $[\lambda]$.

Let $C$ be a quasi-homogeneous curve determined by $f=0$, where $f$ is a
reduced polynomial. Then Lemma \ref{general decomp.} says that $f$ can be
(uniquely) written in the form%
\[
f(x,y)=x^{m}y^{k}%
{\displaystyle\prod\limits_{j=1}^{n}}
(y^{p}-\lambda_{j}x^{q})
\]
where $m,k\in\mathbb{Z}_{2}$, $p,q\in\mathbb{Z}_{+}$, $p\leq q$, $\gcd
(p,q)=1$, and $\lambda_{j}\in\mathbb{C}^{\ast}$ are pairwise distinct. In
particular $C$ has $n+m+k$ distinct branches. Since the exceptional divisor of
the standard resolution and the number of irreducible components are analytic
invariants of a germ of curve, then Lemmas \ref{companion fibration} and
\ref{companion 2} ensure that the triple $(p,q,n)$ is an analytic invariant of
the curve. Thus we have to consider the following three distinct cases:

\begin{enumerate}
\item[i)] $f(x,y)=x^{m}%
{\displaystyle\prod\limits_{j=1}^{n}}
(y-\lambda_{j}x)$ where $m\in\mathbb{Z}_{2}$, and $\lambda_{j}\in\mathbb{C}$.

\item[ii)] $f(x,y)=x^{m}%
{\displaystyle\prod\limits_{j=1}^{n}}
(y-\lambda_{j}x^{q})$ where $m\in\mathbb{Z}_{2}$, $q\in\mathbb{Z}_{+}$,
$q\geq2$ and $\lambda_{j}\in\mathbb{C}$.

\item[iii)] $f(x,y)=x^{m}y^{k}%
{\displaystyle\prod\limits_{j=1}^{n}}
(y^{p}-\lambda_{j}x^{q})$ where $m,k\in\mathbb{Z}_{2}$, $p,q\in\mathbb{Z}_{+}%
$, $2\leq p<q$, $\gcd(p,q)=1$, and $\lambda_{j}\in\mathbb{C}^{\ast}$.
\end{enumerate}

A quasi-homogeneous curve is said to be of \textit{type} $(1,1,n)$, $(1,q,n)$,
and $(p,q,n)$ respectively in cases i), ii), and iii).
\vglue.1in
\noindent{\bf Theorem A}
{\sl The analytic moduli space of germs of quasi-homogeneous curves of
type $(p,q,n)$ are given respectively by\medskip

\begin{enumerate}
\item[i)] $\frac{\operatorname*{Symm}(\mathbb{P}_{\Delta}^{1}(n))}%
{\operatorname*{PSL}(2,\mathbb{C})}$ if $(p,q)=(1,1)$;\bigskip

\item[ii)] $\mathbb{Z}_{2}\times\frac{\operatorname*{Symm}(\mathbb{C}_{\Delta
}(n))}{\operatorname*{Aff}(\mathbb{C})}$ if $p=1$ and $q>1$;\bigskip

\item[iii)] $\mathbb{Z}_{2}\times\mathbb{Z}_{2}\times\frac
{\operatorname*{Symm}(\mathbb{C}_{\Delta}^{\ast}(n))}{\operatorname*{GL}%
(1,\mathbb{C})}$ if $1<p<q$.
\end{enumerate}
}

\vglue.1in
\section{Quasi-homogeneous polynomials}

\subsection{Normal forms}

A\ quasi-homogeneous polynomial $f\in\mathbb{C}[x,y]$ is called
\textit{commode}\textbf{ }if its Newton polygon intersects both coordinate
axis. Further, notice that a polynomial in two variables $P\in\mathbb{C}[x,y]$
may be considered as a polynomial in the variable $y$ with coefficients in
$\mathbb{C}[x]$, i.e. $P\in(\mathbb{C}[x])[y]$. Let $\operatorname*{ord}_{y}P$
be the order of $P$ as a polynomial in $(\mathbb{C}[x])[y]$. Similarly let
$\operatorname*{ord}_{x}P$ be the order of $P$ as an element of $(\mathbb{C}%
[y])[x]$. Therefore, a quasi-homogeneous polynomial $P\in\mathbb{C}[x,y]$ is
commode if and only if $\operatorname*{ord}_{x}P=\operatorname*{ord}_{y}P=0$.
Next, we recall the general behavior of a quasi-homogeneous polynomial.

\begin{lemma}
\label{first decomp.}Let $P\in\mathbb{C}[x,y]$ be a quasi-homogeneous
polynomial, then it has a unique decomposition in the form
\[
P(x,y)=x^{m}y^{n}P_{0}(x,y)
\]
where $m,n\in\mathbb{N}$, $\lambda\in\mathbb{C}$, and $P_{0}$ is a commode
quasi-homogeneous polynomial.
\end{lemma}

\begin{proof}
Let $m:=\operatorname*{ord}_{x}P$ and $n:=\operatorname*{ord}_{y}P$. Clearly,
both $x^{m}$ and $y^{n}$ divide $P$. Hence $P$ can be written in the form
$P(x,y)=\sum_{ai+bj=d}a_{ij}x^{i}y^{j}$ where $i\geq m$ and $j\geq n$. Thus
$P(x,y)=x^{m}y^{n}P_{0}(x,y)$ where $P_{0}(x,y)=\sum_{ai^{\prime}+bj^{\prime
}=d^{\prime}}a_{i^{\prime}+m,j^{\prime}+n}x^{i^{\prime}}y^{j^{\prime}}$ and
$d^{\prime}:=d-am-bn$. Since $m=\operatorname*{ord}_{x}P$ and
$n=\operatorname*{ord}_{y}P$, then $\operatorname*{ord}_{x}P_{0}%
=0=\operatorname*{ord}_{y}P_{0}$. The result then follows directly from the
above remark.
\end{proof}

\begin{definition}
A commode polynomial $P\in\mathbb{C}[x,y]$ is called \textit{monic} in $y$ if
it is a monic polynomial in $(\mathbb{C}[x])[y]$.
\end{definition}

\begin{lemma}
\label{main decomp.}Let $P\in\mathbb{C}[x,y]$ be a commode quasi-homogeneous
polynomial, which is monic in $y$. Then $P$ can be written uniquely as
\[
P(x,y)=\prod_{\ell=1}^{k}(y^{p}-\lambda_{\ell}x^{q})\text{,}%
\]
where $\gcd(p,q)=1$ and $\lambda_{\ell}\in\mathbb{C}^{\ast}$.
\end{lemma}

\begin{proof}
First remark that any quasi-homogeneous polynomial can be written in the form
$P(x,y)=\sum_{pi+qj=m}a_{ij}x^{i}y^{j}$ where $p,q,m\in\mathbb{N}$ and
$\gcd(p,q)=1$. Since $P$ is commode, there are $i_{0},j_{0}\in\mathbb{N}$ such
that $qj_{0}=m$ and $pi_{0}=m$; in particular $k:=m/pq\in\mathbb{N}$.
Therefore $pi+qj=pqk$. Since $\gcd(p,q)=1$, then $q$ divides $i$ and $p$
divides $j$. If we let $i=qi^{\prime}$ and $j=pj^{\prime}$, then $pqi^{\prime
}+qpj^{\prime}=pqk$. Thus $P$ can be written in the form $P(x,y)=\sum
_{i+j=k}a_{qi,pj}x^{qi}y^{pj}$. Let $y=tx^{\frac{q}{p}}$, then the above
equation assumes the form $P(x,tx^{q/p})=x^{qk}\sum_{i+j=k}a_{qi,pj}t^{pj}$.
Now let $\{\lambda_{j}\}_{j=1}^{k}$ be the roots of the polynomial
$g(z)=\sum_{i+j=k}a_{qi,pj}z^{j}$, then
\begin{align*}
P(x,y)  &  =x^{qk}\prod_{\ell=1}^{k}(t^{p}-\lambda_{l})=x^{qk}\prod_{\ell
=1}^{k}(\frac{y^{p}}{x^{q}}-\lambda_{l})\\
&  =\prod_{\ell=1}^{k}(y^{p}-\lambda_{l}x^{q}).
\end{align*}

\end{proof}

\begin{lemma}
\label{general decomp.}Let $P\in\mathbb{C}[x,y]$ be a quasi-homogeneous
polynomial. Then $P$ can be written, uniquely, in the form%
\[
P(x,y)=\mu x^{m}y^{n}\prod_{\ell=1}^{k}(y^{p}-\lambda_{\ell}x^{q})
\]
where $m,n,p,q\in\mathbb{N}$, $\mu,\lambda_{\ell}\in\mathbb{C}^{\ast}$, and
$\gcd(p,q)=1$.
\end{lemma}

\begin{proof}
In view of Lemma \ref{first decomp.} and Lemma \ref{main decomp.}, it is
enough to remark that any commode quasi-homogeneous polynomial $P\in
\mathbb{C}[x,y]$ can be written uniquely as $P=\mu P_{0}$ where $P_{0}$ is
monic in $y$.
\end{proof}

\subsection{Resolution}

We recall the geometry of the exceptional divisor of the minimal resolution of
a germ of quasi-homogeneous curve.

A \textit{tree of projective lines} is an embedding of a connected and simply
connected chain of projective lines intersecting transversely in a complex
surface (two dimensional complex analytic manifold) with two projective lines
in each intersection. In fact, it consists of a pasting of Hopf bundles whose
zero sections are the projective lines themselves. A\textbf{\ }\textit{tree of
points} is any tree of projective lines in which a finite number of points is
discriminated. The above nomenclature has a natural motivation. In fact, as is
well know, we can assign to each projective line a point and to each
intersection an edge in other to form the \textit{weighted dual graph}. Two
trees of points are called \textit{isomorphic} if their weighted dual graph
are isomorphic (as graphs). \ It is well known that any germ of analytic curve
$C$ in $(\mathbb{C}^{2},0)$ has a standard resolution, which we denote by
$\widetilde{C}$. If the exceptional divisor of $\widetilde{C}$ has just one
projective line containing three or more singular points of $\widetilde{C}$,
then it is called the \textit{principal projective line} of $\widetilde{C}$
and denoted by $D_{\operatorname*{pr}(\widetilde{C})}$. A tree of projective
lines is called a \textit{linear chain} if each of its projective lines
intersects at most other two projective lines of the tree. A\ projective line
of a linear chain is called an \textit{end} if it intersects just another one
projective line of the chain.

\begin{lemma}
\label{companion fibration}Let $C$ be a commode quasi-homogeneous curve. Then
its standard resolution tree is a linear chain and its standard resolution
$\widetilde{C}$ intersects just one projective line of $D$, i.e. $C$ has one
of the following diagrams\ of resolution:%

{\includegraphics[
trim=0.000000in 0.000000in 0.004544in 0.000000in,
height=0.5725in,
width=2.3653in
]%
{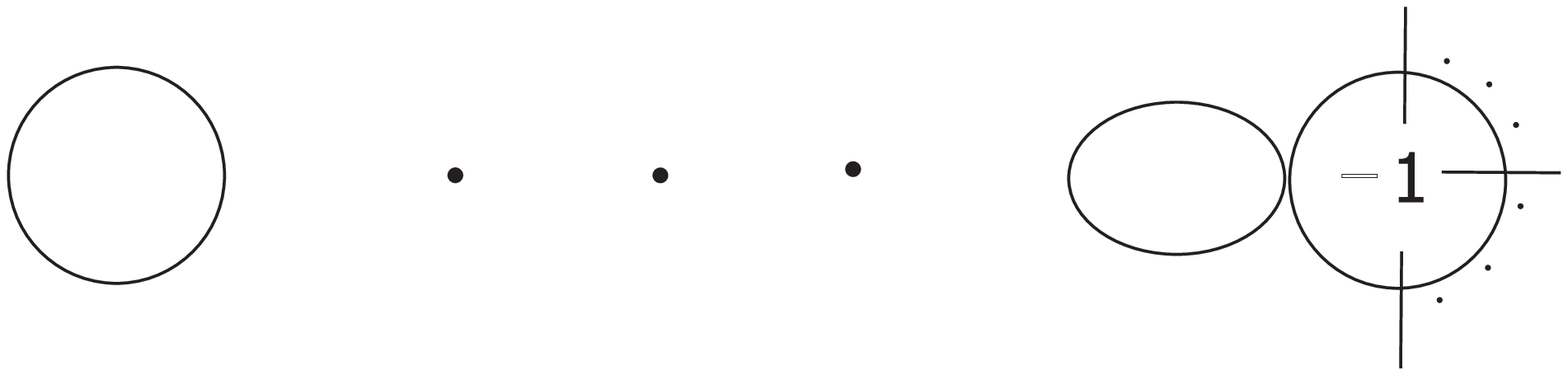}%
}
%

{\includegraphics[
trim=0.000000in 0.000000in -0.001454in 0.000000in,
height=0.5483in,
width=2.3367in
]%
{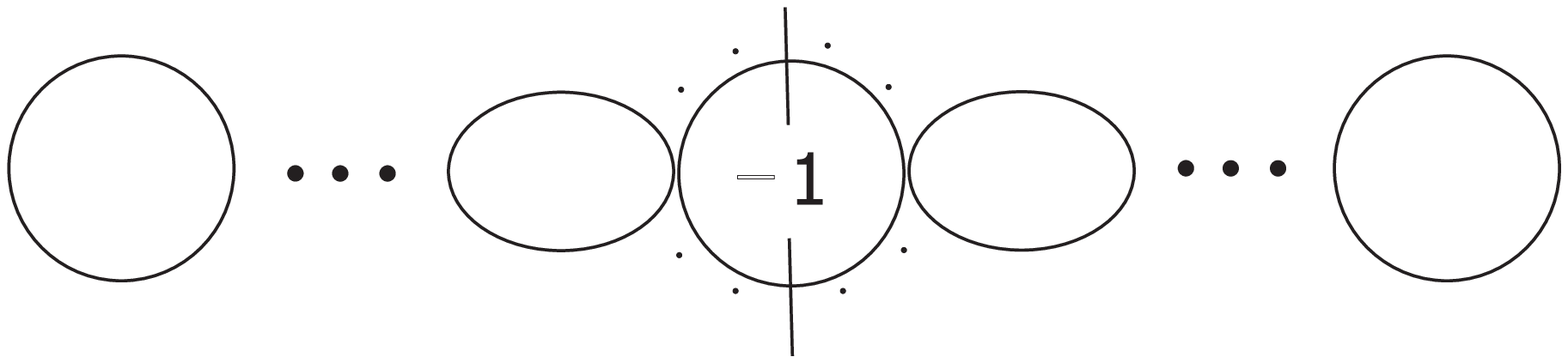}%
}

\end{lemma}

\begin{proof}
From Lemma \ref{main decomp.}, there is a local system of coordinates $(x,y)$
such that $C=f^{-1}(0)$ where $f(x,y)=\prod_{l=1}^{k}(y^{p}-\lambda_{j}x^{q})$
with $p<q$ and $\gcd(p,q)=1$. Since each irreducible curve $y^{p}-\lambda
_{l}x^{q}=0$ is a generic fiber of the fibration $\frac{y^{p}}{x^{q}}%
\equiv\operatorname*{const}$, then it is resolved together with the fibration.
After one blowup we obtain:
\[%
\begin{array}
[c]{c}%
t^{p}/x^{q-p}\equiv\operatorname*{const},\\
u^{q}y^{q-p}\equiv\operatorname*{const}.
\end{array}
\]
Since $p<q$, we have a singularity with holomorphic first integral at infinity
and a meromorphic first integral at the origin (as before). Going on with this
process, Euclid's algorithm assures that the resolution ends after the blowup
of a radial foliation. In particular, if $p=1$, then it is easy to see that
the\ principal projective line is transversal to just one projective line of
the divisor. Otherwise (i.e. if $p\neq1$) the singularity with meromorphic
first integral \textquotedblleft moves\textquotedblright\ to the
\textquotedblleft infinity\textquotedblright, i.e. it will appear in a corner
singularity. Then the principal projective line intersects exactly two
projective lines of the divisor.
\end{proof}

Let $\#\operatorname*{irred}(\widetilde{C})$ denote the number of irreducible
components of $\widetilde{C}$.

\begin{lemma}
\label{companion 2}Let $C$ be a non-commode quasi-homogeneous curve. Then its
minimal resolution tree is a linear chain having a principal projective line
such that $\#(\widetilde{\mathcal{C}}\cap D_{\operatorname*{pr}(\widetilde
{\mathcal{C}})})\leq$ $\#\operatorname*{irred}(\widetilde{\mathcal{C}})-1$.
Further $\widetilde{\mathcal{C}}\cap D_{j}=\emptyset$ whenever $D_{j}$ is
neither the principal projective line nor an end; i.e. $\mathcal{C}$ has one
of the following diagrams\ of resolution:%

{\includegraphics[
height=0.5621in,
width=2.3359in
]%
{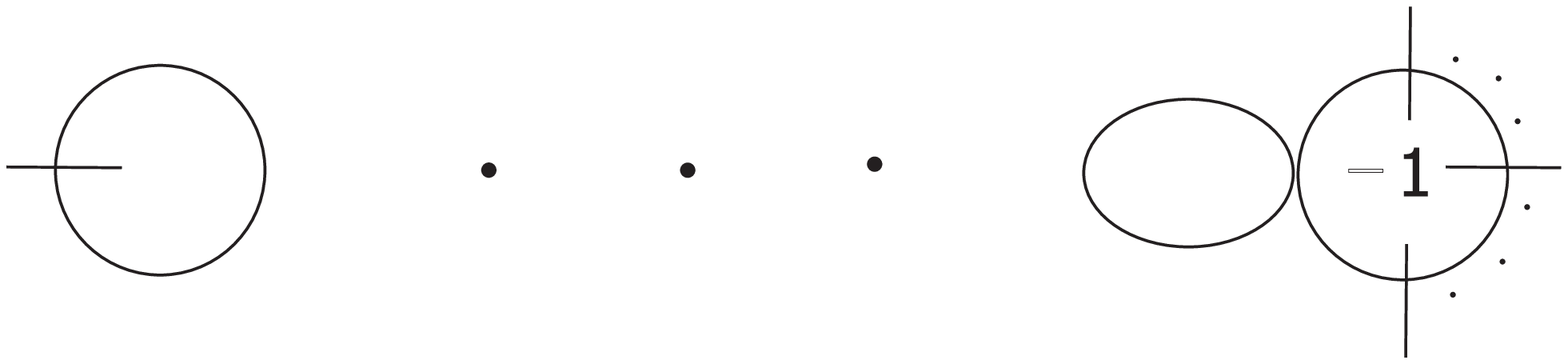}%
}
%

{\includegraphics[
trim=0.000000in 0.000000in -0.005302in 0.000000in,
height=0.5224in,
width=2.316in
]%
{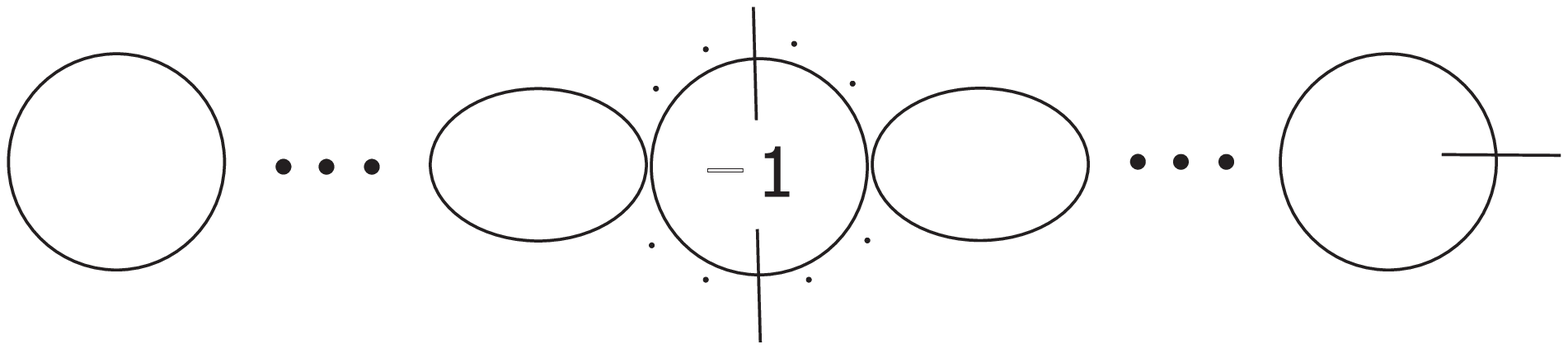}%
}
%

{\includegraphics[
trim=0.000000in 0.000000in 0.003124in 0.000000in,
height=0.512in,
width=2.3359in
]%
{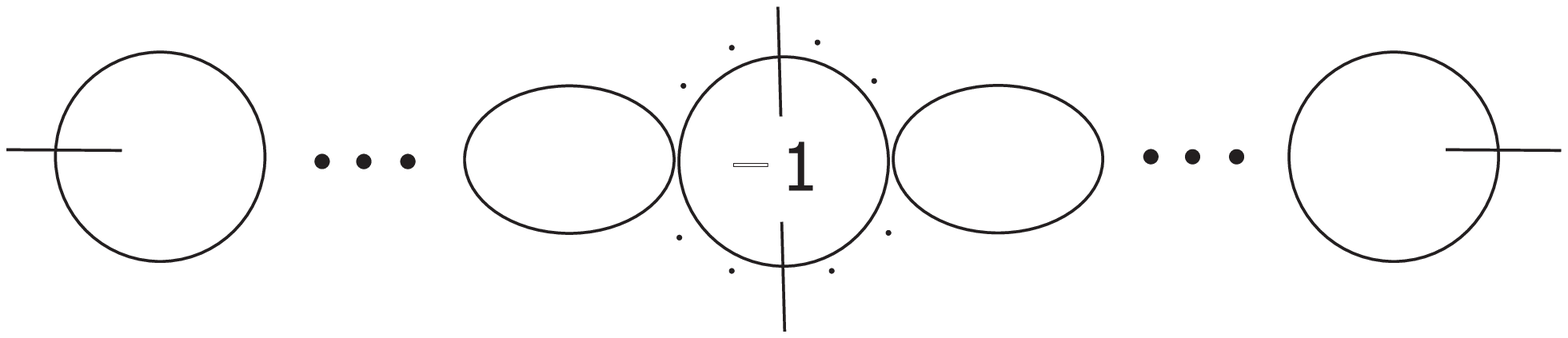}%
}

\end{lemma}

\begin{proof}
From Lemma \ref{general decomp.}, there is a local system of coordinates
$(x,y)$ such that $\mathcal{C}=f^{-1}(0)$ where $f(x,y)=\mu x^{m}y^{n}%
\prod_{l=1}^{k}(y^{p}-\lambda_{j}x^{q})$, $p<q$, and $\gcd(p,q)=1$. Since $\mu
x^{m}y^{n}$ is resolved after one blowup, then $f(x,y)$ is resolved together
with the fibration $\frac{y^{p}}{x^{q}}\equiv\operatorname*{const}$, as
before. Then the result follows from Lemma \ref{companion fibration}.
\end{proof}

\section{Quasi-homogeneous curves}

We consider each case separately and prove Theorem ~A in a series of lemmas.

\subsection{Curves of type $(1,1,n)$.}

In this case the curve is given as the zero set of a polynomial of the form
$f(x,y)=x^{m}%
{\displaystyle\prod\limits_{j=1}^{n}}
(y-\lambda_{j}x)$ where $m\in\mathbb{Z}_{2}$, and $\lambda_{j}\in\mathbb{C}$;
in particular it is resolved after one blowup. Thus, given $\lambda
=(\lambda_{1},\cdots,\lambda_{n})\in\mathbb{P}_{\Delta}^{1}(n)$ we define
$f_{\lambda}(x,y)=x%
{\displaystyle\prod\limits_{j\neq i}}
(y-\lambda_{j}x)$ if $\lambda_{i}=\infty$ or $f_{\lambda}(x,y)=%
{\displaystyle\prod\limits_{j=1}^{n}}
(y-\lambda_{j}x)$ if $\lambda_{j}\neq\infty$ for all $j=1,\ldots,k$. We denote
the curve $f_{\lambda}=0$ by $C_{\lambda}$. Recall that the natural action of
$\operatorname*{PSL}(2,\mathbb{C})$ in $\mathbb{P}^{1}$ as the group of
homographies induces a natural action of $\operatorname*{PSL}(2,\mathbb{C})$
in $\operatorname*{Symm}(\mathbb{P}_{\Delta}^{1}(n))$. Further, recall that
the equivalence class of $\lambda\in$ $\mathbb{P}_{\Delta}^{1}(n)$ in
$\operatorname*{Symm}(\mathbb{P}_{\Delta}^{1}(n))/\operatorname*{PSL}%
(2,\mathbb{C})$ is denoted by $[\lambda]$.

\begin{lemma}
\label{(1,1)}Two homogeneous curves $C_{\lambda}$ and $C_{\mu}$ are
analytically equivalent if and only if $[\lambda]=[\mu]\in\operatorname*{Symm}%
(\mathbb{P}_{\Delta}^{1}(n))/\operatorname*{PSL}(2,\mathbb{C})$.
\end{lemma}

\begin{proof}
Suppose $C_{\lambda}$ and $C_{\mu}$ are analytically equivalent and let
$\Phi\in Dif(\mathbb{C}^{2},0)$ take $C_{\lambda}$ into $C_{\mu}$. Let
$\widetilde{\Phi}$ be the blowup of $\Phi$, then it takes the strict transform
of $C_{\lambda}$ into the strict transform of $C_{\mu}$. Blowing up
$f_{\lambda}$ and $f_{\mu}$ we obtain at once that the first tangent cones of
$C_{\lambda}$ and $C_{\mu}$ are respectively given by $\{\lambda_{1}%
,\cdots,\lambda_{n}\}$ and $\{\mu_{1},\cdots,\mu_{n}\}$. Therefore, there is
$\sigma\in S_{n}$ such that the M\"{o}bius transformation $\varphi=\left.
\widetilde{\Phi}\right\vert _{\mathbb{P}^{1}}$ satisfies $\mu_{\sigma
(j)}=\varphi(\lambda_{j})$ for all $j=1,\ldots,n$. In other words
$[\lambda]=[\mu]$. Conversely, suppose $[\lambda]=[\mu]$. Reordering the
indexes of $\{\mu_{1},\cdots,\mu_{n}\}$ we may suppose, without loss of
generality, that there is a M\"{o}bius transformation $\varphi(z)=\frac
{az+b}{cz+d}$, with $ad-bc=1$, such that $\mu_{j}=\varphi(\lambda_{j})$ for
all $j=1,\ldots,n$. Now consider the linear transformation
$T(x,y)=(dx+cy,bx+ay)$ with inverse $T^{-1}(x,y)=(ax-cy,-bx+dy)$. Then a
straightforward calculation shows that $f_{\lambda}=\alpha\cdot T^{\ast}%
f_{\mu}$ where $\alpha\in\mathbb{C}^{\ast}$. Thus $C_{\lambda}$ is
analytically equivalent to $C_{\mu}$, as desired.
\end{proof}

\begin{remark}
\label{three branches}Recall that for any three distinct points $\{\lambda
_{1},\lambda_{2},\lambda_{3}\}\subset\mathbb{P}^{1}$ there is a M\"{o}bius
transformation $\varphi$ such that $\varphi(0)=\lambda_{1}$, $\varphi
(1)=\lambda_{2}$ and $\varphi(\infty)=\lambda_{3}$.
\end{remark}

As a straightforward consequence of Lemma \ref{(1,1)} and Remark
\ref{three branches} one has:

\begin{corollary}
Let $\lambda,\mu\in\mathbb{P}_{\Delta}^{1}(n)$ with $n\leq3$. Then
$C_{\lambda}$ and $C_{\mu}$ are analytically equivalent.
\end{corollary}

\subsection{Curves of type $(1,q,n)$, $q\geq2$.}

In this case, the curve is given as the zero set of a polynomial of the form
$f_{m,\lambda}(x,y)=x^{m}%
{\displaystyle\prod\limits_{j=1}^{n}}
(y-\lambda_{j}x^{q})$ where $m\in\mathbb{Z}_{2}$, $q\in\mathbb{Z}_{+}$,
$q\geq2$, and $\lambda_{j}\in\mathbb{C}$. Thus given $m\in\mathbb{Z}_{2}$ and
$\lambda=(\lambda_{1},\cdots,\lambda_{n})\in\mathbb{C}_{\Delta}(n)$, we denote
a curve of type $(1,q,n)$ by $C_{m,\lambda}$ if it is given as the zero set of
$f_{m,\lambda}$. Recall that the group of affine transformations of
$\mathbb{C}$, denoted by $\operatorname*{Aff}(\mathbb{C})$, acts in a natural
way in $\operatorname*{Symm}(\mathbb{C}_{\Delta}(n))$. Further, recall that
the equivalence class of $\lambda\in$ $\mathbb{C}_{\Delta}(n)$ in
$\operatorname*{Symm}(\mathbb{C}_{\Delta}(n))/\operatorname*{Aff}(\mathbb{C})$
is denoted by $[\lambda]$.

\begin{lemma}
\label{(1,q)}Two homogeneous curves $C_{m,\lambda}$ and $C_{m,\mu}$ are
analytically equivalent if and only if $[\lambda]=[\mu]\in\operatorname*{Symm}%
(\mathbb{C}_{\Delta}(n))/\operatorname*{Aff}(\mathbb{C})$.
\end{lemma}

\begin{proof}
Suppose $\Phi\in Diff(\mathbb{C}^{2},0)$ is an equivalence between
$C_{m,\lambda}$ and $C_{m,\mu}$. From the proof of Lemma
\ref{companion fibration}, both curves are resolved after $q$ blowups.
Further, after $q-1$ blowups $\Phi$ will be lifted to a local conjugacy
$\Phi^{(q-1)}$ between the germs of curves given in local coordinates $(x,y)$
respectively by $p_{\lambda}(x,y)=x%
{\displaystyle\prod\limits_{j=1}^{n}}
(y-\lambda_{j}x)$ and $p_{\mu}(x,y)=x%
{\displaystyle\prod\limits_{j=1}^{n}}
(y-\mu_{j}x)$ where $(x=0)$ is the local equation of the exceptional divisor
$D^{(q-1)}$. Let $\pi$ denote a further blowup given in local coordinates by
$\pi(t,x)=(x,tx)$ and $\pi(u,y)=(u,uy)$, and $\Phi^{(q)}$ be the map obtained
by the lifting of $\Phi^{(q-1)}$ by $\pi$. Further, let $\varphi=\left.
\Phi^{(q)}\right\vert _{D_{q}}$ where $D_{q}=\pi^{-1}(0)$. Since $\Phi^{(q)}$
preserves the irreducible components of $\pi^{\ast}(D^{(q-1)})$, then
$\varphi(t)=\Phi^{(q)}(t,0)$ is a homography fixing $\infty$ and conjugating
the first tangent cones of $p_{\lambda}=0$ and $p_{\mu}=0$ respectively. Thus
$[\lambda]=[\mu]\in\operatorname*{Symm}(\mathbb{C}_{\Delta}%
(n))/\operatorname*{Aff}(\mathbb{C})$. Conversely, (reordering the indexes of
$\mu$, if necessary) suppose there is $\varphi(z)=az+b\in\operatorname*{Aff}%
(\mathbb{C})$ such that $\mu_{j}=\varphi(\lambda_{j})$ for all $j=1,\ldots,n$,
and let $T(x,y)=(x,ay+bx^{q})$. Then a straightforward calculation shows that
$f_{m,\lambda}=\alpha\cdot T^{\ast}f_{m,\mu}$ where $\alpha\in\mathbb{C}%
^{\ast}$. Thus $C_{m,\lambda}$ and $C_{m,\mu}$ are analytically equivalent, as desired.
\end{proof}

As a straightforward consequence of Lemma \ref{(1,q)} and Remark
\ref{three branches} one has:

\begin{corollary}
Let $\lambda,\mu\in\mathbb{C}_{\Delta}(n)$ with $n\leq2$. Then $C_{m,\lambda}$
and $C_{m,\mu}$ are analytically equivalent.
\end{corollary}

\subsection{Curves of type $(p,q,n)$, $2\leq p<q$.}

In this case, the curve is given as the zero set of a polynomial of the form
$f_{m,k,\lambda}(x,y)=x^{m}y^{k}%
{\displaystyle\prod\limits_{j=1}^{n}}
(y^{p}-\lambda_{j}x^{q})$ where $m,k=0,1$, $p,q\in\mathbb{Z}_{+}$, $2\leq
p<q$, and $\lambda_{j}\in\mathbb{C}^{\ast}$. Thus given $\lambda=(\lambda
_{1},\cdots,\lambda_{n})\in\mathbb{C}_{\Delta}^{\ast}(n)$ we denote a curve of
type $(p,q,n)$ by $C_{m,k,\lambda}$ if it is given as the zero set of
$f_{m,k,\lambda}(x,y)$. Recall that the group of linear transformations of
$\mathbb{C}$, denoted by $\operatorname*{GL}(1,\mathbb{C})$, acts in a natural
way in $\operatorname*{Symm}(\mathbb{C}_{\Delta}^{\ast}(n))$. Further, recall
that the equivalence class of $\lambda\in$ $\mathbb{C}_{\Delta}^{\ast}(n)$ in
$\operatorname*{Symm}(\mathbb{C}_{\Delta}^{\ast}(n))/\operatorname*{GL}%
(1,\mathbb{C})$ is denoted by $[\lambda]$.

\begin{lemma}
\label{(p,q)}Two homogeneous curves $C_{m,k,\lambda}$ and $C_{m,k,\mu}$ are
analytically equivalent if and only if $[\lambda]=[\mu]\in\operatorname*{Symm}%
(\mathbb{C}_{\Delta}^{\ast}(n))/\operatorname*{GL}(1,\mathbb{C})$.
\end{lemma}

\begin{proof}
First recall from the proof of Lemma \ref{companion fibration} that
$C_{m,k,\lambda}$ is resolved after $N$ blowups, where $N$ depends on the
Euclid's division algorithm between $q$ and $p$. Further, in the $(N-1)^{th}$
step we have to blowup a singularity given in local coordinates $(x,y)$ as the
zero set of the polynomial $g_{\lambda}(x,y)=xy%
{\displaystyle\prod\limits_{j=1}^{n}}
(y-\lambda_{j}x)$. Therefore, if $\Phi\in Diff(\mathbb{C}^{2},0)$ is an
equivalence between $C_{m,k,\lambda}$ and $C_{m,k,\mu}$ and $\Phi^{(N-1)}$ is
its lifting to the $(N-1)^{th}$ step of the resolution, then it conjugates the
germs of curves given in local coordinates $(x,y)$ respectively by
$p_{\lambda}(x,y)=xy%
{\displaystyle\prod\limits_{j=1}^{n}}
(y-\lambda_{j}x)$ and $p_{\mu}(x,y)=xy%
{\displaystyle\prod\limits_{j=1}^{n}}
(y-\mu_{j}x)$ where $(x=0)$ and $(y=0)$ are local equations for the
exceptional divisor $D^{(N-1)}$. Let $\pi$ denote the final blowup of the
resolution given in local coordinates by $\pi(t,x)=(x,tx)$ and $\pi
(u,y)=(u,uy)$, and $\Phi^{(N)}$ be the map obtained by the lifting of
$\Phi^{(N-1)}$ by $\pi$. Further let $\varphi=\left.  \Phi^{(N)}\right\vert
_{D_{N}}$ where $D_{N}=\pi^{-1}(0)$. Since $\Phi^{(N)}$ preserves the
irreducible components of $\pi^{\ast}(D^{(q-1)})$, then $\varphi(t)=\Phi
^{(q)}(t,0)$ is a homography fixing $0$ and $\infty$, and conjugating the
first tangent cones of $p_{\lambda}=0$ and $p_{\mu}=0$ respectively. Thus
$[\lambda]=[\mu]\in\operatorname*{Symm}(\mathbb{C}_{\Delta}^{\ast
}(n))/\operatorname*{GL}(1,\mathbb{C})$. Conversely, (reordering the indexes
of $\mu$, if necessary) suppose there is $\varphi(z)=az\in\operatorname*{GL}%
(1,\mathbb{C})$ such that $\mu_{j}=\varphi(\lambda_{j})$ for all
$j=1,\ldots,n$, and let $T(x,y)=(x,\sqrt[p]{a}y)$. Then a straightforward
calculation shows that $f_{m,\lambda}=\alpha\cdot T^{\ast}f_{m,\mu}$ where
$\alpha\in\mathbb{C}^{\ast}$. Thus $C_{m,\lambda}$ and $C_{m,\mu}$ are
analytically equivalent, as desired.
\end{proof}

As a straightforward consequence of Lemma \ref{(p,q)} and Remark
\ref{three branches} one has:

\begin{corollary}
Let $\lambda,\mu\in\mathbb{C}_{\Delta}^{\ast}(1)$, then $C_{m,k,\lambda}$ and
$C_{m,k,\mu}$ are analytically equivalent.
\end{corollary}

\section{Resolution and factorization}

We study the relationship between the resolution tree and the factorization of
a quasi-homogeneous polynomial. We use the resolution in order to study the
equivalence between two quasi-homogeneous polynomials.

First recall that a quasi-homogeneous polynomial split uniquely in the form
$P=x^{m}y^{n}P_{0}$ where $P_{0}$ is a commode quasi-homogeneous polynomial.
In particular $P$ and $P_{0}$ share the same resolution process.

\begin{corollary}
Let $P\in\mathbb{C}[x,y]$ be a commode quasi-homogeneous polynomial with the
weights $(p,q)$, where $\gcd(p,q)=1$. Let $q_{j}=s_{j}p_{j}+r_{j}$,
$j=1,\ldots,m$, be the Euclid's algorithm of $(p,q)$, where $q_{1}:=q$,
$p_{1}:=p$, $q_{j+1}:=p_{j}$, and $p_{j+1}:=r_{j}$ for all $j=1,\ldots,m-1$.
Then the exceptional divisor of its minimal resolution is given by a linear
chain of projective lines, namely $D=\cup_{j=1}^{n}D_{j}$, whose
self-intersection numbers are given as follows:

\begin{enumerate}
\item If $m=2\alpha-1$, then
\[
D_{j}\cdot D_{j}=\left\{
\begin{array}
[c]{cl}%
-(s_{2k}+2) & \text{if }j=s_{1}+\cdots+s_{2k-1}\text{, }k=1,\ldots
,\alpha-1\text{;}\\
-1 & \text{if }j=s_{1}+\cdots+s_{2\alpha-1}\text{;}\\
-(s_{2k+1}+2) & \text{if }j=m-(s_{2}+\cdots+s_{2k-2})+1\text{, }%
k=1,\ldots,\alpha-1\text{;}\\
-(s_{2\alpha-1}+1) & \text{if }j=m-(s_{1}+\cdots+s_{2\alpha-2})+1\text{;}\\
-2 & \text{otherwise.}%
\end{array}
\right.
\]

\item If $m=2\alpha$, then
\[
D_{j}\cdot D_{j}=\left\{
\begin{array}
[c]{cl}%
-(s_{2k}+2) & \text{if }j=s_{1}+\cdots+s_{2k-1}\text{, }k=1,\ldots
,\alpha-1\text{;}\\
-(s_{2\alpha}+1) & \text{if }j=s_{1}+\cdots+s_{2\alpha-1}\text{;}\\
-(s_{2k+1}+2) & \text{if }j=m-(s_{2}+\cdots+s_{2k-2})+1\text{, }%
k=1,\ldots,\alpha-1\text{;}\\
-1 & \text{if }j=m-(s_{1}+\cdots+s_{2\alpha-2})+1\text{;}\\
-2 & \text{otherwise.}%
\end{array}
\right.
\]
Finally, if $C$ is given by $f=0$ where $f(x,y)=x^{m}y^{n}%
{\displaystyle\prod\limits_{j=1}^{k}}
(y^{p}-\lambda_{j}x^{q})$, then a representative of $[\lambda]$ is determined
by the intersection of the strict transform of $C$\ with the exceptional
divisor $D$.
\end{enumerate}
\end{corollary}

\begin{proof}
The proof shall be performed by induction on $m$, the length of the Euclidean
algorithm. In order to better understand the arguments, the reader have to
keep in mind the proof of Lemma \ref{companion fibration}. From Lemma
\ref{main decomp.}, we may suppose, without loss of generality, that $P$ can
be written in the form $P(x,y)=%
{\displaystyle\prod\limits_{j=1}^{k}}
(y^{p}-\lambda_{j}x^{q})$. First remark that if $m=1$ then $p=1$. Thus we
prove the statement for $m=1$ by induction on $q$. For $q=1$ the result is
easily verified after one blowup. Now suppose the result is true for all
$q\leq q_{0}-1$. Then after one blowup $\pi(t,x)=(x,tx)$, $\pi(u,y)=(uy,y)$,
$P$ is transformed into $\pi^{\ast}P(t,x)=x%
{\displaystyle\prod\limits_{j=1}^{k}}
(t-\lambda_{j}x^{q-1})$. Thus the result follows for $m=1$ by induction on
$q$. Suppose the result is true for all polynomials whose pair of weights have
Euclid's algorithm length less than $m$, and let $(p,q)$ has length $m$. Since
$p_{j}=s_{j}q_{j}+r_{j}$, $j=1,\ldots,m$, is the Euclid's algorithm of
$(p,q)$, then $p_{j}=s_{j}q_{j}+r_{j}$, $j=2,\ldots,m$, is the Euclid's
algorithm of $(p_{2},q_{2})$. In particular the Euclid's algorithm of
$(p_{2},q_{2})$ has length $m-1$. Reasoning in a similar way as in the case
$m=1$, we have after $s_{1}$ blowups a linear chain of projective lines
$\cup_{j=1}^{s_{1}}D_{j}^{(1)}$ such that $D_{j}^{(1)}\cdot D_{j}^{(1)}=-2$
for all $j=1,\ldots,s_{1}-1$ and $D_{s1}^{(1)}\cdot D_{s_{1}}^{(1)}=-1$.
Further, the strict transform of $P=0$ is given by the zero set of the
polynomial $\widetilde{P}(t,x)=%
{\displaystyle\prod\limits_{j=1}^{k}}
(t^{p_{1}}-\lambda_{j}x^{r_{1}})=$ $\lambda_{1}\cdots\lambda_{k}%
{\displaystyle\prod\limits_{j=1}^{k}}
(x^{p_{2}}-\lambda_{j}t^{q_{2}})$ where the local equation for $D_{s_{1}%
}^{(1)}$ is $(x=0)$. The first statement thus follows by the induction
hypothesis. The last statement comes immediately from the above reasoning.
\ For the above induction arguments ensure that the strict transform of $P$
assume the form $\widetilde{P}=0$, with $\widetilde{P}(x,y)=%
{\displaystyle\prod\limits_{j=1}^{k}}
(y-\lambda_{j}x)$, just before the last blowup.
\end{proof}

The above Corollary gives an easy way to compute the relatively prime weights
of a quasi-homogeneous polynomials from the dual weighted tree of its minimal
resolution. Also it shows that the minimal resolution can be used both to
split a quasi-homogeneous polynomial into irreducible components and also to
determine its analytic type.

\part{Classification of foliations}

\section{Preliminaries}

A germ of a singular foliation $(\mathcal{F}:\omega=0)$ in $(\mathbb{C}^{2},0)$
of codimension $1$ is, roughly speaking, the set of integral curves of a given
germ of $1$-form $\omega\in\Omega^{1}(\mathbb{C}^{2},0)$, which may be assumed
to have just an isolated singularity at the origin. Let $\operatorname*{Diff}%
(\mathbb{C}^{k},0)$ be the group of germs of analytic diffeomorphisms of
$(\mathbb{C}^{k},0)$ fixing the origin. Two germs of foliations $(\mathcal{F}%
_{j}:\omega_{j}=0)$ in $(\mathbb{C}^{2},0)$, $j=1,2$, are analytically
equivalent if there is $\Phi\in\operatorname*{Diff}(\mathbb{C}^{2},0)$ sending
leaves of $\mathcal{F}_{1}$ into leaves of $\mathcal{F}_{2}$. One says that
$h_{1},h_{2}\in\operatorname*{Diff}(\mathbb{C},0)$ are analytically conjugate
if there is $\phi\in\operatorname*{Diff}(\mathbb{C},0)$ such that $Ad_{\phi
}(h_{1}):=\phi\circ h_{1}\circ\phi^{-1}=h_{2}$. We denote the \textit{Hopf
bundle\ of order }$k$ (see Definition \ref{Hopf bundle}) by $p_{(k)}%
:\mathcal{H}(-k)\rightarrow D$ where $D\simeq\mathbb{CP}(1)$, or just by its
total space $\mathbb{H}(-k)$. Let $\pi:(\widetilde{X},D)\longrightarrow
\mathbb{(C}^{2},0\mathbb{)}$ be the map resulting from the iteration of a
finite number of blowups with exceptional divisor $D=\pi^{-1}(0)$. Let
$D=\cup$ $D_{j}$ be its decomposition into irreducible components where
$D_{j}$ has self-intersection number equal to $-k_{j}$ for $j=1,\ldots,n$.
Then recall from the theory of algebraic curves that a suitable neighborhood
of $D$ in $\widetilde{X}$ results from pasting together suitable neighborhoods
of the zero sections of $\mathbb{H}(-k_{j})$. We denote by $\widetilde
{\mathcal{F}}$ the unique extension of $\pi^{\ast}(\mathcal{F})$ whose
singular set has codimension greater or equal to $2$ (cf. \cite{MaMo 80}). For
each Hopf bundle $p_{j}:\mathcal{H}_{j}\rightarrow D_{j}$ of a given
resolution, we denote by $\widetilde{\mathcal{F}}_{j}$ the germ of a foliation in $(\mathcal{H}_{j},D_{j})$ induced by the restriction of $\widetilde
{\mathcal{F}}$ and call it the $j^{\text{th}}$ \textit{Hopf component} of the
resolution. The singular points of the exceptional divisor, namely
$c_{ij}:=D_{i}\cap D_{j}$, are called \textit{corners} and the singularities
about such points are called \textit{corner singularities} (or just corners)
and denoted by $\widetilde{\mathcal{F}}_{ij}$. The \textquotedblleft strict
transform\textquotedblright\ of $\operatorname*{Sep}(\mathcal{F})$ at
$D_{j}\subset$ $\mathcal{H}_{j}$, i.e. the set of local separatrices of
$\widetilde{\mathcal{F}}_{j}$, namely $\operatorname*{Sep}(\widetilde
{\mathcal{F}}_{j})=\overline{\left(  \pi^{\ast}\operatorname*{Sep}%
(\mathcal{F})\right)  |_{\mathcal{H}_{j}}\backslash D_{j}}$, is called the
$j^{\text{th}}$ \textit{Hopf component}\textbf{ }of $\pi^{\ast}%
(\operatorname*{Sep}(\mathcal{F}))$. Two foliations having analytically
equivalent Hopf components are called \textit{analytically componentwise
equivalent}.

Let $p:\mathcal{H}\rightarrow D$ be a Hopf bundle and $\mathcal{F}$ a germ of a foliation defined in $(\mathcal{H},D)$. Then $\mathcal{F}$ is called
\textit{non-dicritical} if $D$ is an invariant set of $\mathcal{F}$, and
\textit{dicritical} otherwise. In the former case the holonomy of
$\mathcal{F}$ with respect to $D$ evaluated at a transversal section $\Sigma$
is called the \textit{projective holonomy} of $\mathcal{F}$ and denoted by
$\operatorname*{Hol}_{\Sigma}(\mathcal{F},D)$. One says that $\mathcal{F}$ is
\textit{resolved} if it has just \textit{reduced} singularities (cf.
\cite{MaMo 80}). Let $\widetilde{\mathcal{F}}^{1}$ and $\widetilde
{\mathcal{F}}^{2}$ be two germs of non-dicritical singular foliations at
$D\subset\mathcal{H}$ without saddle-nodes and having the same singular set,
say $\{t_{j}\}_{j=1}^{n}$. Let $t_{0}\in D$ be a regular point of\ $\widetilde
{\mathcal{F}}^{1}$ and denote by $h_{\gamma}^{i}$ the holonomy of a path
$\gamma\in\pi_{1}(D\backslash\{t_{j}\}_{j=1}^{n},t_{0})$ with respect to $D$
evaluated at a transversal section $\Sigma_{0}:=p^{-1}(t_{0})$. Then one says
that the projective holonomies of these foliations are \textit{analytically
conjugate }if there is $\phi\in\operatorname*{Diff}(\mathbb{C},0)$ such that
$Ad_{\phi}(h_{\gamma}^{1})=h_{\gamma}^{2}$ for every $\gamma\in\pi
_{1}(D\backslash\{t_{j}\}_{j=1}^{n},t_{0})$.

A \textit{generalized curve} is a germ of a singular foliation in $(\mathbb{C}%
^{2},0)$ that has no saddle-node or dicritical components along its minimal
resolution (cf. \cite{CaLNSa 84}). A germ of  a holomorphic function
$f\in\mathbb{C}\{x,y\}$ is said to be \textit{quasi-homogeneous} if there is a
local system of coordinates in which $f$ can be represented by a
quasi-homogeneous polynomial, i.e. $f(x,y)=\sum_{ai+bj=d}a_{ij}x^{i}y^{j}$
where $a,b,d\in\mathbb{N}$. A\ quasi-homogeneous polynomial $f\in
\mathbb{C}[x,y]$ is called \textit{commode}\textbf{ }if its Newton polygon
intersects both coordinate axis. The separatrix set of a germ of a foliation $\mathcal{F}$ in $(\mathbb{C}^{2},0)$ is said to be quasi-homogeneous if
$\operatorname*{Sep}(\mathcal{F})=f^{-1}(0)$ where $f$ is a quasi-homogeneous
function. The set of generalized curves in $(\mathbb{C}^{2},0)$ with
quasi-homogeneous separatrix set is denoted by $\mathcal{QHS}$; in particular, if
$\operatorname*{Sep}(\mathcal{F})$ is commode, then $\mathcal{F}$ is called a
\textit{commode} $\mathcal{QHS}$\ foliation.

A \textit{tree of projective lines} is an embedding of a connected and simply
connected chain of projective lines intersecting transversely in a complex
surface (two dimensional complex analytic manifold) with two projective lines
in each intersection. In fact, it consists of the pasting of Hopf bundles
whose zero sections are the projective lines themselves. A\textbf{\ }%
\textit{tree of points} is any tree of projective lines in which are
discriminated a finite number of points. The above nomenclature has a natural
motivation. In fact, as is well know, we can assign to each projective line a
point and to each intersection an edge in other to form the \textit{weighted
dual graph}. Two trees of points are called \textit{isomorphic} if their
weighted dual graph are isomorphic (as graphs).
\begin{center}
\includegraphics[
height=0.9695in,
width=3.0381in
]%
{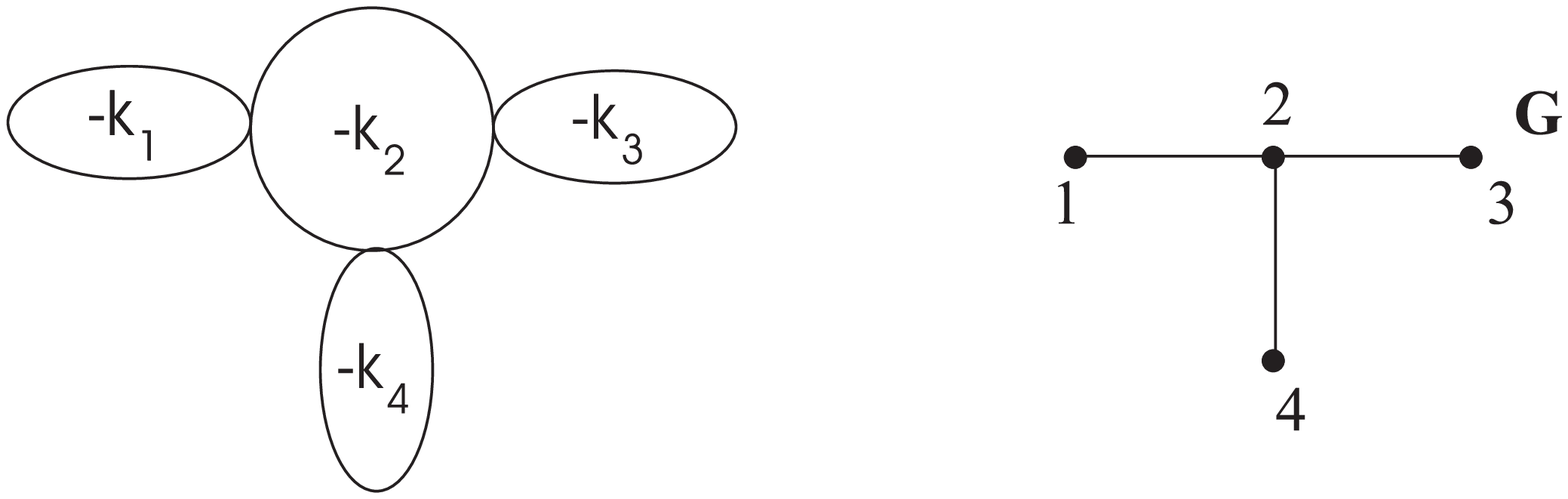}%
\\
Figure 1
\end{center}

Recall, from \cite{Se 68}, that any germ of a holomorphic foliation
$\mathcal{F}$ in $(\mathbb{C}^{2},0)$ has a minimal resolution. We denote it
by $\widetilde{\mathcal{F}}$ and its ambient surface by $M_{\widetilde
{\mathcal{F}}}$. If the exceptional divisor of $\widetilde{\mathcal{F}}$ has
just one projective line containing three or more singular points of
$\widetilde{\mathcal{F}}$, then it is called the \textit{principal projective
line} of $\widetilde{\mathcal{F}}$ and denoted by $D_{\operatorname*{pr}%
(\widetilde{\mathcal{F}})}$. If $\widetilde{\mathcal{F}}$ has a principal
projective line, then the projective holonomy of its principal projective line
is called the \textit{projective holonomy of the foliation $\mathcal{F}$}. Later on, we will
see that any $\mathcal{QHS}$ foliation has a principal projective line. Then
one says that $\mathcal{F}\in\mathcal{QHS}$ is \textit{generic}\textbf{ }if
the singularities of $\widetilde{\mathcal{F}}$ about the corners of $D$ in
$D_{\operatorname*{pr}(\widetilde{\mathcal{F}})}$ are in the
Poincar\'{e}-Dulac or Siegel domain (cf. \cite{Ar 80}).

\vglue.1in
\noindent{\bf Theorem B} {\sl Let $\mathcal{F}$ and $\mathcal{F}^{\prime}$ be two
$\mathcal{QHS}$ germs of foliations with the same separatrix set. Suppose that
$\mathcal{F}$ and $\mathcal{F}^{\prime}$ are both commode or generic. Then
$\mathcal{F}$ and $\mathcal{F}^{\prime}$ are analytically equivalent if and
only if their projective holonomies are analytically conjugate.}

\vglue.1in

\section{Hopf bundles and projective holonomy}

We consider non-dicritical resolved\ singular foliations without saddle-nodes
defined in a neighborhood of the zero section of a Hopf bundle. Under some
natural geometric conditions, we describe the invariants that determine their
analytic type.

First recall the definition of a Hopf bundle.

\begin{definition}
\label{Hopf bundle}Let $k\in\mathbb{Z}_{+}$ and consider two copies of
$\mathbb{C}^{2}$ with coordinates given respectively by $(t,x)$ and $(u,y)$.
Then the line bundle over $\mathbb{CP}(1)$ given by the transition maps%
\[
\left\{
\begin{array}
[c]{c}%
y=t^{k}x\\
u=1/t
\end{array}
\right.
\]
for all $t\neq0$ is called the Hopf bundle of order $k$ and denoted by
$p_{(k)}:\mathcal{H}(-k)\rightarrow\mathbb{CP}(1)$ or just by its total space
$\mathbb{H}(-k)$.
\end{definition}

Clearly, two analytically equivalent singularities have isomorphic weighted
dual trees of singular points along their minimal resolution. Thus, if we
consider analytically equivalent Hopf components, it is clear that isomorphic
points have the same linear part and that their local holonomy generators are
conjugated by a global map. To clarify the ideas, we need the following

\begin{definition}
\label{transversality}Let $M$ be a complex surface and $S\subset M$ a smooth
curve invariant by the germ of a holomorphic foliation $\mathcal{F}$ in $(M,S)$
such that $\operatorname*{Sing}(\mathcal{F})\subset S$ has just
non-degenerated reduced singularities (i.e. without saddle nodes). Then we say
that a germ of a holomorphic map $f:(M,S)\rightarrow S$ is a fibration
transversal to $\mathcal{F}$ if it satisfies:

\begin{enumerate}
\item \label{item 1}$f$ is a retraction, i.e. $f$ is a submersion and
$f|_{S}=\operatorname*{id}|_{S}$;

\item \label{item 2}the fiber $f^{-1}(t_{j})$ is a separatrix of $\mathcal{F}$
for each $t_{j}\in\operatorname*{Sing}(\mathcal{F})$;

\item \label{item 3}$f^{-1}(t)$ is transversal to $\mathcal{F}$ for every
(regular) point $t\in S\backslash\operatorname*{Sing}(\mathcal{F})$.
\end{enumerate}
\end{definition}

Let $\mathcal{F}$ be a germ of a singular holomorphic foliation without
saddle-nodes defined in a neighborhood of the zero section of the Hopf bundle
$p:\mathcal{H}\rightarrow D$, $f:(\mathcal{H},D)\rightarrow D$ a fibration
transversal to $\mathcal{F}$, and $t_{0}\in D\backslash\operatorname*{Sing}%
(\mathcal{F})$ a regular point of $\mathcal{F}$. Hence the path lifting
construction ensures that the projective holonomy $\operatorname*{Hol}%
_{f^{-1}(t)}(\mathcal{F},D)$ is completely determined by $\operatorname*{Hol}%
_{f^{-1}(t_{o})}(\mathcal{F},D)$ for any $t,t_{0}\in D\backslash
\operatorname*{Sing}(\mathcal{F})$. Such a holonomy is called the
\textit{projective holonomy} of $\mathcal{F}$ with respect to $f$. If there is
no doubt about the fibration, we only talk about the projective holonomy of
the foliation and denote it by $\operatorname*{Hol}(\mathcal{F},D)$.

\begin{definition}
Let $\mathcal{F}$ and $\mathcal{F}_{o}$ be germs of singular resolved and
non-dicritical foliations defined in a neighborhood of the zero section of the
Hopf bundle $p:\mathcal{H}\rightarrow D$ with the same singular set $S$. Then
we set
\[
\operatorname*{Diff}\nolimits_{\mathcal{F},\mathcal{F}_{o}}(\mathcal{H}%
,D):=\{\Phi\in\operatorname*{Diff}(\mathcal{H},D):\,\Phi_{\ast}(\mathcal{F}%
)=\mathcal{F}_{o}\text{ and }\,\Phi|_{S}=\operatorname*{id}\}
\]
and call
\[
\operatorname*{Aut}(\mathcal{F}_{o}):=\{\Phi\in\operatorname*{Diff}%
\nolimits_{\mathcal{F}_{o},\mathcal{F}_{o}}(\mathcal{H},D):\Phi|_{S}%
=\operatorname*{id}\}
\]
the group of automorphisms of $\mathcal{F}_{o}$. Further, if $f:(\mathcal{H}%
,D)\rightarrow D$ is a fibration transversal to $\mathcal{F}_{o}$, then the
set of elements of $\operatorname*{Aut}(\mathcal{F}_{o})$ preserving $f$ is
denoted by $\operatorname*{Aut}(\mathcal{F}_{o},f)$.
\end{definition}

\begin{proposition}
\label{holonomy lifting}Let $\mathcal{F}^{i}$, $i=1,2$, be two germs of
resolved and non-dicritical singular foliations without saddle-nodes defined
in a neighborhood of the zero section of the Hopf bundle $p:\mathcal{H}%
\rightarrow D$. Suppose that $\operatorname*{Sep}(\mathcal{F}^{1}%
)=\operatorname*{Sep}(\mathcal{F}^{2})$ and that there is a fibration
$f_{i}:(\mathcal{H},D)\longrightarrow D$ transversal to $\mathcal{F}^{i}$.
Then $\mathcal{F}^{1}$ and $\mathcal{F}^{2}$ are analytically equivalent if
and only if their projective holonomies are analytically conjugate.
\end{proposition}

\begin{proof}
As already remarked, the necessary part is straightforward. Let us treat the
sufficient part. Since the separatrices of $\mathcal{F}^{1}$ and
$\mathcal{F}^{2}$ coincide, then their singular sets also coincide. Let
$\operatorname*{Sing}(\mathcal{F}^{i})=\{t_{j}\}_{j=1}^{n}$ and $t_{0}\in D$
be a regular point. Suppose there is $\phi\in\operatorname*{Diff}%
(\mathbb{C},0)$ such that $\phi\circ(h_{j}^{1})\circ\phi^{-1}=h_{j}^{2}$ for
all $j=1,\ldots,n$. Then define the map $\Phi:\mathcal{F}\backslash
\bigcup_{j=1}^{n}f_{1}^{-1}(t_{j})\longrightarrow\mathcal{F}^{\prime
}\backslash\bigcup_{j=1}^{n}f_{2}^{-1}(t_{j})$ by%
\[
\Phi(t,x):=\Phi_{t}(x):=h_{t}^{2}\circ\phi\circ(h_{t}^{1})^{-1}(x),
\]
where $x\in f_{1}^{-1}(t)$ and $h_{t}^{i}:f_{i}^{-1}(t_{0})\longrightarrow
f_{i}^{-1}(t)$ are the holonomy maps obtained by path lifting a curve
connecting $t_{0}$ to $t$ along the leaves of $\mathcal{F}^{i}$. Note that
this map does not depend on the chosen base curves, since $\phi$ conjugates
the elements of the respective projective holonomies of $\mathcal{F}^{1}$ and
$\mathcal{F}^{2}$. Since $\Phi$ is holomorphic in each variable separately,
then (complex) ODE theory and Hartogs' theorem ensure that $\Phi$ is
holomorphic. Finally, since $\mathcal{F}^{1}$ has just reduced singularities,
then \cite{MaMo 80},\cite{MaRa 83} ensure that the union of the saturated of
$\Sigma_{0}:=f_{1}^{-1}(t_{0})$ along the leaves of $\mathcal{F}^{1}$ and the
local separatrices $\operatorname*{Sep}(\mathcal{F}^{1})=\bigcup_{j=1}%
^{n}f_{1}^{-1}(t_{j})$ gives rise to a neighborhood of $D$. Thus we can use
Riemann's extension theorem in order to extend $\Phi$ to $\operatorname*{Sep}%
(\mathcal{F}^{1})$ in a neighborhood of $D$.
\end{proof}

\section{Analytic invariants\label{anal. class. sec.}}

We consider germs of foliations in $(\mathbb{C}^{2},0)$ and use the weighted
dual trees of their minimal resolutions, the first jet of each singularity of
these resolutions, and the projective holonomies of their Hopf components in
order to determine analytic componentwise equivalence. Next, we identify some
analytical cocycles that appear as obstructions to extend analytically
componentwise isomorphism. Finally, we relate these obstructions with the
analytic classification of the foliations.

\subsection{Componentwise equivalence and realization}

We find conditions to determine whether two $\mathcal{QHS}$ foliations with
the same quasi-homogeneous separatrix set are componentwise equivalent.

First, let us introduce some notation. Let $\mathcal{QHS}_{f}$ denote the set
of $\mathcal{QHS}$ foliations with the same separatrix set $f=0$. Let
$\mathcal{F},\mathcal{F}^{\prime}\in\mathcal{QHS}_{f}$ and $\widetilde
{\mathcal{F}},\widetilde{\mathcal{F}}^{\prime}$ be respectively their minimal
resolutions. Let $\operatorname*{Sing}(\widetilde{\mathcal{F}})=\{t_{i,j_{i}%
}\}_{i,j_{i}=1}^{k,n_{i}}$ where $k$ is the number of Hopf components of
$\widetilde{\mathcal{F}}$ and $n_{i}:=\#\operatorname*{Sing}(\widetilde
{\mathcal{F}}_{i})$. Let $\omega_{i,j_{i}}=0$ and $\omega_{i,j_{i}}^{\prime
}=0$ determine the germs of $\widetilde{\mathcal{F}}$ and $\widetilde
{\mathcal{F}}^{\prime}$ at $t_{i,j_{i}}$. Then one says that $\widetilde
{\mathcal{F}}^{\prime}$ is \textit{analytically componentwise equivalent} to
$\widetilde{\mathcal{F}}$ \textit{up to first order} if $J^{1}(\omega
_{i,j_{i}})=J^{1}(\omega_{i,j_{i}}^{\prime})$ (i.e. if they have the same
linear part) for all $i=1,\ldots,k$ and $j_{i}=1,\ldots,n_{i}$. The set of
$\mathcal{QHS}_{f}$ foliations analytically componentwise equivalent up to
first order to $(\mathcal{F}:\omega=0)$ is denoted by $\mathcal{QHS}%
_{\omega,f}^{c,1}$. Finally, denote the set of $\mathcal{QHS}$ (respect.
$\mathcal{QHS}_{f}$) foliations analytically componentwise equivalent to
$(\mathcal{F}:\omega=0)$ by $\mathcal{QHS}_{\omega}^{c}$ (respect.
$\mathcal{QHS}_{\omega,f}^{c}$).

\begin{remark}
Any element in $\mathcal{QHS}_{f}$ has its weighted dual graph automatically
determined by the separatrix $f=0$.
\end{remark}

We determine now the moduli space $\left.  \mathcal{QHS}_{\omega}%
^{c,1}\right/  \mathcal{QHS}_{\omega}^{c}$. The following result is a
straightforward consequence of Proposition \ref{holonomy lifting}.

\begin{proposition}
\label{first reduction}Let $\mathcal{F}$ and $\mathcal{F}^{\prime}$ belong
with the same conjugacy class in $\mathcal{QHS}_{\omega}^{c,1}$. Then they
belong with the same conjugacy class in $\mathcal{QHS}_{\omega}^{c}$ if and
only if their projective holonomies are analytically conjugate.
\end{proposition}

Given two germs of foliations in $\mathcal{QHS}_{\omega}^{c}$, we want to
verify under what conditions they are in fact globally holomorphically
conjugate. For this sake, we need the following realization data.

\begin{definition}
\label{resolution-like}A complex surface is called resolution-like if it is
obtained by a holomorphic pasting of Hopf bundles with negative Chern classes,
in such a way that the union of their zero sections become a tree of
projective lines isomorphic to the exceptional divisor of a composition of a
finite numbers of blowups applied to $(\mathbb{C}^{2},0)$.
\end{definition}

Clearly, this definition is given in such a way that every resolution surface
of some singularity is automatically resolution-like. In fact, any
resolution-like surface is biholomorphic to the resolution surface of some singularity.

\begin{proposition}
[\cite{Ca 01}]\label{surface unicity}Let $M$ be a resolution-like surface with
tree of projective lines $D$. Then $(M,D)$ can be realized as a neighborhood
of the exceptional divisor of a composition of a finite number of blowups
applied to $(\mathbb{C}^{2},0)$.
\end{proposition}

In order to prove this proposition, we need the following results about
complex line bundles.

\begin{theorem}
[Grauert \cite{Gra 62}]Let $S$ be a complex surface and $C\subset S$ be a
rational curve with negative self-intersection number. Then there are
neighborhoods $U$ and $V$ of $C$, respectively in $S$ and $N(C;S)$ (the normal
bundle of $C$ in $S$), and a biholomorphism $\Psi:U\rightarrow V$ sending $C$
in the zero section of $N(C;S)$.
\end{theorem}

\begin{theorem}
[Grothendieck \cite{Gro 57}]Two complex line bundles over the Riemann sphere
have the same Chern class if and only if they are biholomorphic.
\end{theorem}

\begin{proof}
[Proof of Proposition \ref{surface unicity}]The proof is performed by
induction on the number of projective lines in the chain. If the chain is
composed by just one projective line, the result follows immediately from the
theorems of Grauert and Grothendieck. Suppose the result is true for all
chains composed by $n\geq1$ projective lines and let $D_{j}$ have $n+1$
projective lines. From the hypothesis, $D_{j}$ has two intersecting projective
lines, namely $C_{j}^{1}$ and $C_{j}^{2}$, with self-intersection numbers
given respectively by $-1$ and $-2$. Hence, applying Grauert's and
Grothendieck's theorems, we obtain that a neighborhood of each curve is
biholomorphic to a neighborhood of the zero section of the Hopf bundle with
Chern classes given by their self-intersection numbers. Thus we can blow down
a neighborhood of the curve $C_{j}^{1}$ obtaining yet an analytic surface
defined in a neighborhood of a Riemann sphere, say $\pi(C_{j}^{2})$ --- where
$\pi$ stands for the blow down. Since $\pi(C_{j}^{2})$ is smooth, it is well
known that its self-intersection number is $-1$ (cf. e.g. \cite{Lau 71}). The
result now follows from the induction hypothesis.
\end{proof}

\begin{remark}
Although two foliations in $\mathcal{QHS}_{\omega}^{c}$ are not necessarily
defined in the same ambient surface, they all can be modeled by $(\mathcal{F}%
:\omega=0)$ in the sense that they are analytically componentwise equivalent
to $\mathcal{F}$. Anyway, the ambient surface will be automatically equivalent
whenever they have equivalent cocycles (definition found below).
\end{remark}

\subsection{Analytic cocycles\label{anal. cocycles}}

We construct some cocycles associated with analytically componentwise
equivalent foliations. In some sense, these cocycles measure how far two
analytically componentwise equivalent foliations are from being analytically equivalent.

Let $\mathcal{F}^{o}\in\mathcal{QHS}$, $\widetilde{\mathcal{F}}^{o}$ its
minimal resolution, and $M^{o}=M_{\widetilde{\mathcal{F}}^{o}}$ the ambient
surface where $\widetilde{\mathcal{F}}^{o}$ is defined. Let
$\operatorname*{Pseudo}(M^{o})$ denote the pseudogroup of transformations of
$M^{o}$ and $\operatorname*{Aut}(\widetilde{\mathcal{F}}^{o})$ denote its
subset given by those $\phi\in\operatorname*{Pseudo}(M^{o})$ satisfying the
following properties:

\begin{enumerate}
\item[(a)] $\phi:U\longrightarrow\phi(U)$ preserves the Hopf components of the
exceptional divisor, i.e. $\phi(U\cap D_{j})=\phi(U)\cap$ $D_{j}$;

\item[(b)] $\phi$ fixes the singularities of $\widetilde{\mathcal{F}}^{o}$,
i.e. $\left.  \phi\right\vert _{\operatorname*{Sing}(\widetilde{\mathcal{F}%
}^{o})}=\left.  \operatorname*{id}\right\vert _{\operatorname*{Sing}%
(\widetilde{\mathcal{F}}^{o})}$;

\item[(c)] $\phi$ preserves the leaves of $\widetilde{\mathcal{F}}_{j}^{o}$,
i.e. $\phi^{\ast}(\left.  \widetilde{\mathcal{F}}_{j}^{o}\right\vert
_{\phi(U)})=\left.  \widetilde{\mathcal{F}}_{j}^{o}\right\vert _{U}$.
\end{enumerate}

At this point, some comments about the above definition are worthwhile. First,
notice that all conditions can be verified explicitly. The first two are quite
obvious and the third can be achieved with the aid of the path lifting
procedure. In fact, choose a section $\Sigma$ transversal to $D_{j}$ and pick
an element $\psi:\phi(\Sigma)\longrightarrow\Sigma$ of the classical holonomy
pseudogroup of $\left.  \widetilde{\mathcal{F}}_{j}^{o}\right\vert _{U}$ with
respect to $D_{j}$. Since the holonomy characterizes $\left.  \widetilde
{\mathcal{F}}_{j}^{o}\right\vert _{U}$ (cf. Proposition \ref{holonomy lifting}%
, \cite{MaMo 80}, \cite{MaRa 83}), it is enough to verify that $\psi\circ
\phi\in\operatorname*{Diff}(\Sigma)$ commutes with the generators of
$\operatorname*{Hol}_{\Sigma}(\left.  \widetilde{\mathcal{F}}_{j}%
^{o}\right\vert _{U},D_{j})$. Further, note that we decided to deal with just
local and semilocal leaves (i.e. those determined by the holonomies of
$\left.  \widetilde{\mathcal{F}}_{j}^{o}\right\vert _{U}$) avoiding, for the
time been, questions related with Dulac maps (cf. \cite{CaSc 2001}, \cite{CaSc
2003}) that are very difficult to handle concretely in the global sense. This
task will be performed by the pasting cocycles we define next.

\begin{definition}
Let $(\mathcal{F}:\omega=0)$ be a germ of a foliation in $(\mathbb{C}^{2},0)$.
Then the set%
\[
\operatorname*{Aut}(\mathcal{F})=\{\phi\in\operatorname*{Diff}(\mathbb{C}%
^{2},0):\phi^{\ast}\omega\wedge\omega=0\}
\]
is called the group of automorphisms of $\mathcal{F}$. Further, if
$f:(M,S)\rightarrow S$ is a fibration transversal to $\mathcal{F}$, then
$\operatorname*{Aut}(\mathcal{F},f)$ denote the subgroup determined by
elements of $\operatorname*{Aut}(\mathcal{F})$ preserving $f$.
\end{definition}

Let $(\mathcal{F}:\omega=0)$ be a generalized curve and pick $\mathcal{F}^{o}$
analytically componentwise equivalent to $\mathcal{F}$ such that
$\operatorname*{Sep}(\widetilde{\mathcal{F}}_{j}^{o})$ consists of fibers of a
fibration $f_{j}:(\mathcal{H}_{j},D_{j})\longrightarrow D_{j}$ transversal to
$\widetilde{\mathcal{F}}_{j}^{o}$ (such a resolution exists from \cite{Ca
01}). Then $\mathcal{F}^{o}$ is called a \textit{fixed model} for
$\mathcal{F}$ and a map $\Phi_{j}\in\operatorname*{Diff}(\widetilde
{\mathcal{F}}_{j},\widetilde{\mathcal{F}}_{j}^{o})$ is called a
\textit{projective chart}\textbf{ }for $\widetilde{\mathcal{F}}$ with respect
to $\widetilde{\mathcal{F}}^{o}$. From we have done before, it is
straightforward that:

\begin{lemma}
\label{project. charts}For each $\widetilde{\mathcal{F}}_{j}=\left.
\widetilde{\mathcal{F}}\right\vert _{(\mathcal{H}_{j},D_{j})}$ and each fixed
model component $\widetilde{\mathcal{F}}_{j}^{o}$, there exists only one
projective chart up to left composition with an element of
$\operatorname*{Aut}(\widetilde{\mathcal{F}}_{j}^{o})$.
\end{lemma}

Let $D=\cup D_{j}$ be the exceptional divisor of $\widetilde{\mathcal{F}}^{o}%
$. One says that $\mathcal{U}:=\cup U_{j}$ is a \textit{good covering} for $D$
if each\ $U_{j}$ is a simply-connected neighborhood of $D_{j}\subset
\mathcal{H}_{j}$ and each intersections $U_{i}\cap U_{j}$ is simply-connected.
For each good covering $\mathcal{U}$ and each foliation $\mathcal{F}$ one can
associate a cocycle $\Phi(\mathcal{F}):=(\Phi_{i,j})$ given by $\Phi
_{i,j}:=\Phi_{i}\circ\Phi_{j}^{-1}$ where each $\Phi_{i}$ is a projective
chart for $\widetilde{\mathcal{F}}$ with respect to $\widetilde{\mathcal{F}%
}^{o}$. Note that $(\Phi_{i,j})$ does not depend neither on the fixed models
nor on the chosen (good) covering up to analytically componentwise equivalence class.

\begin{proposition}
\label{theta}Two analytically componentwise equivalent generalized curves
$\mathcal{F}$ and $\mathcal{G}$ are analytically equivalent if and only if
$\Phi(\mathcal{F})=\Phi(\mathcal{G})$.
\end{proposition}

\begin{proof}
Let $\Phi(\mathcal{F})=(\Phi_{1}\circ\Phi_{2}^{-1},\cdots,\Phi_{k-1}\circ
\Phi_{k}^{-1})$ and $\Phi(\mathcal{G})=(\Psi_{1}\circ\Psi_{2}^{-1},\cdots
,\Psi_{k-1}\circ\Psi_{k}^{-1})$. First, let us verify the necessary part.
Suppose $H$ is a global conjugation between $\mathcal{F}$ and $\mathcal{G}$,
i.e. $H^{\ast}(\mathcal{G})=\mathcal{F}$. From Lemma \ref{project. charts},
there is $\Xi_{j}\in\operatorname*{Aut}(\widetilde{\mathcal{F}}_{j}^{o})$ such
that $\Psi_{j}=\Xi_{j}\circ\Phi_{j}\circ H$. Therefore%
\begin{align*}
\Psi_{j-1}\circ\Psi_{j}^{-1}  &  =\Xi_{j-1}\circ\Phi_{j-1}\circ H\circ
H^{-1}\circ\Phi_{j}^{-1}\circ\Xi_{j}^{-1}\\
&  =\Xi_{j-1}\circ\Phi_{j-1}\circ\Phi_{j}^{-1}\circ\Xi_{j}^{-1}.
\end{align*}
Now let us verify the sufficient part. Notice that $\mathcal{F}$ and
$\mathcal{G}$ have the same fixed model. Hence, if $(\Phi_{1}\circ\Phi
_{2}^{-1},\cdots,\Phi_{k-1}\circ\Phi_{k}^{-1})=\Phi(\mathcal{F})=\Phi
(\mathcal{G})=(\Psi_{1}\circ\Psi_{2}^{-1},\cdots,\Psi_{k-1}\circ\Psi_{k}%
^{-1})$, there is\ a collection $(\Xi_{j})\subset\operatorname*{Aut}%
(\widetilde{\mathcal{F}}_{j}^{o})$ such that $\Psi_{j-1}\circ\Psi_{j}^{-1}%
=\Xi_{j-1}\circ\Phi_{j-1}\circ\Phi_{j}^{-1}\circ\Xi_{j}^{-1}$. Therefore
$\left(  \Xi_{j-1}\circ\Phi_{j-1}\right)  ^{-1}\circ\Psi_{j-1}=\left(  \Xi
_{j}\circ\Phi_{j}\right)  ^{-1}\circ\Psi_{j}$. Thus we can define a global
conjugation between them just by letting $H:=\left(  \Xi_{j}\circ\Phi
_{j}\right)  ^{-1}\circ\Psi_{j}$ for all $j=1,\ldots,k$.
\end{proof}

\begin{remark}
It is not difficult to verify that $\operatorname*{Aut}(\widetilde
{\mathcal{F}}^{o})$ is itself a pseudogroup of transformations of $M^{o}$.
Therefore the sheaf of germs of elements of $\operatorname*{Aut}%
(\widetilde{\mathcal{F}}^{o})$, generated by inductive limit, is a sheaf of
groupoids over the exceptional divisor $D^{o}$ of $\widetilde{\mathcal{F}}%
^{o}$ (cf. \cite{Ha 58}). We denote this sheaf by $\mathfrak{Aut}%
_{\widetilde{\mathcal{F}}^{o}}$. Consider the first cohomology set
$H^{1}(\mathcal{U},\mathfrak{Aut}$$_{\widetilde{\mathcal{F}}^{o}})$, and let
$H^{1}(D,\mathfrak{Aut}_{\widetilde{\mathcal{F}}^{o}})$ be the inductive limit
of $H^{1}(\mathcal{U},\mathfrak{\operatorname*{Aut}}_{\widetilde{\mathcal{F}%
}^{o}})$ for all good coverings of $D$. Then Proposition \ref{surface unicity}
ensures that the map
\[%
\begin{array}
[c]{ccc}%
\mathcal{QHS}_{\omega}^{c}\! & \overset{\Phi}{\longrightarrow} &
Z^{1}(D,\mathfrak{\operatorname*{Aut}}_{\widetilde{\mathcal{F}}^{o}})\\
\mathcal{F} & \mapsto & (\Phi_{i,j}):=\Phi_{i}\circ\Phi_{j}^{-1}%
\end{array}
\]
is well defined and onto $H^{1}(D,\mathfrak{\operatorname*{Aut}}%
_{\widetilde{\mathcal{F}}^{o}})$. Since $\Phi(\mathcal{F})$ does not depend on
the fixed models up to componentwise equivalence class, it determines a
characteristic class for generalized curves appearing as obstruction for the
global pasting of analytically componentwise isomorphisms. For the reader not
acquainted with groupoids and the cohomology of their sheaves, we refer to
\cite{Eh 61}, \cite{Eh 65} and \cite{Ha 58}.
\end{remark}

\section{Trivializing cocycles\label{vanishing}}

We use the algebraic and geometric features of the separatrix set in order to
construct an auxiliary fibration that helps us to trivialize the cocycles. For
this sake, we have first to introduce the concept of leaf preserving
automorphism. Further, we use the geometry of the divisors of both the
foliation and the fibration in order to provide a method for trivializing
$\Phi(\mathcal{F})$.

\subsection{Quasi-homogeneous polynomials and companion fibrations}

In order to prove Theorem ~B, we need to perform an accurate
geometric analysis of the interplay between the foliation $\mathcal{F}$ and
its companion fibration $\mathcal{G}$.

\subsubsection{Multivalued first integrals and the branches of $\mathcal{F}%
$.\label{coord. syst. and integrals}}

Let $\mathcal{F}\in\mathcal{QHS}_{\omega,f}^{c}$ where
\begin{equation}
f(x,y)=\mu y^{m}x^{n}\prod_{j=1}^{d}(y^{p}-\lambda_{j}x^{q}), \label{eq1}%
\end{equation}
$1\leq p<q,m,n\in\mathbb{N}^{\ast}$, $\gcd(p,q)=1$, and $\lambda_{j},\mu
\in\mathbb{C}^{\ast}$. Then we order the first projective line to arise in the
course of the resolution process with $1$, the next one intersecting it with
$2$, and so on (see Lemma \ref{companion fibration} and Figure 7), until we
reach the last projective line in the minimal resolution. Recall that the
principal projective line is denoted by $D_{\operatorname*{pr}(\widetilde
{\mathcal{F}})}$ or $D_{\ell}$ where $\ell=\operatorname*{pr}(\widetilde
{\mathcal{F}})$. For the sake of simplicity we call the subset of
$\widetilde{\mathcal{F}}$ given by $\mathcal{B}_{+}\mathcal{F}:=\cup_{j>\ell
}\widetilde{\mathcal{F}}_{j}$ (respect. $\mathcal{B}_{-}\mathcal{F}%
:=\cup_{j<\ell}\widetilde{\mathcal{F}}_{j}$) the \textit{positive} (respect.
\textit{negative}) \textit{branch} of $\widetilde{\mathcal{F}}$. We are in a
position to state the following geometric characterization of the branches of
$\widetilde{\mathcal{F}}$.%

\begin{center}
\includegraphics[
trim=0.017000in 0.000000in -0.017000in 0.000000in,
height=0.7092in,
width=3.0381in
]%
{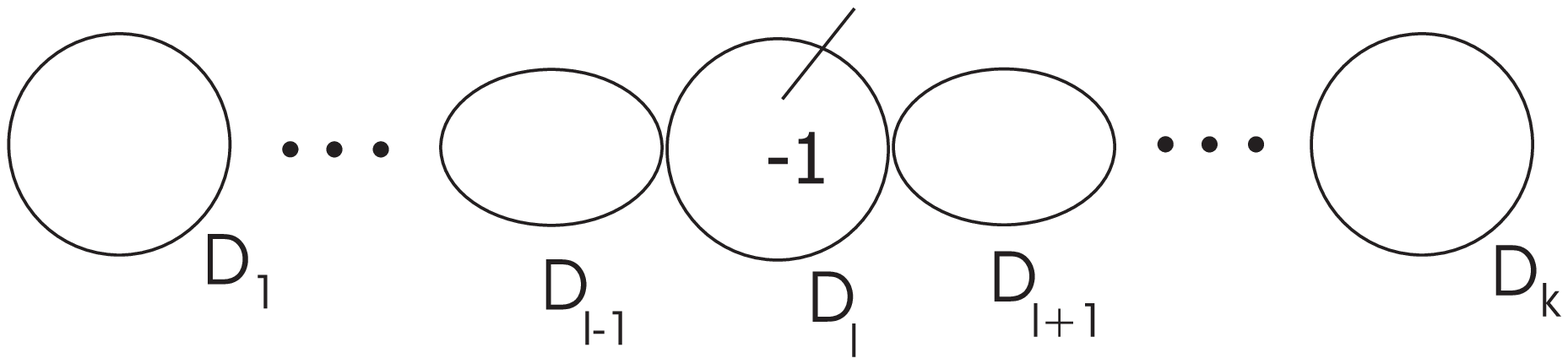}%
\\
Figure 7: The principal projective line for $p\neq1$.
\end{center}

\begin{lemma}
\label{linear models}$\widetilde{\mathcal{F}}_{j}$ is linearizable for each
$j\neq\operatorname*{pr}(\widetilde{\mathcal{F}})$. In particular, it has a
multivalued first integral. More precisely, there is $\Phi_{j}\in
\operatorname*{Diff}_{\widetilde{\mathcal{F}}_{j},\widetilde{\mathcal{F}}%
_{j}^{lin}}(\mathcal{H}_{j},D_{j})$ where $(\widetilde{\mathcal{F}}_{j}%
^{lin}:d\widetilde{f}_{j}^{lin}=0)$ is given by the (global) multivalued first
integral%
\[
\left\{
\begin{array}
[c]{l}%
\widetilde{f}_{j}^{lin}(t_{j},x_{j})=t_{j}^{\nu_{j}}x_{j}^{\mu_{j}},\\
\widetilde{f}_{j}^{lin}(u_{j},y_{j})=u_{j}^{k_{j}\mu_{j}-\nu_{j}}y_{j}%
^{\mu_{j}},
\end{array}
\right.  \text{ }%
\]
where $\nu_{j},\mu_{j}\in\mathbb{C}$ are non-resonant and $-k_{j}$ is the
first Chern class of $\mathbb{H}_{j}$ for all $j\neq\ell$.
\end{lemma}

\begin{proof}
Since $\mathcal{F}$ is generic, then the corner singularities of
$\widetilde{\mathcal{F}}_{\operatorname*{pr}(\widetilde{\mathcal{F}})}$ are
linearizable (cf. \cite{Yoc 88}). But Lemmas \ref{companion fibration} and
\ref{companion 2} ensure that $\widetilde{\mathcal{F}}_{j}$ has at most two
singularities for all $j\neq\operatorname*{pr}(\widetilde{\mathcal{F}})$, thus
both singularities share the same holonomy with respect to $D_{j}$. Recall
from \cite{MaMo 80} that a reduced and non-degenerate (i.e. a non saddle-node)
singularity is linearizable if and only if its holonomy is linearizable. Thus
Proposition \ref{holonomy lifting} ensures that $\widetilde{\mathcal{F}}_{j}$
is linearizable whenever one of its singularities is linearizable.
\end{proof}

Let $(x,y)$ be the system of coordinates about the origin in which
$\operatorname*{Sep}(\mathcal{F})$ assumes the form (\ref{eq1}), then it
induces canonical affine coordinates for $M:=\cup_{j=1}^{n}\mathbb{H}%
_{j}(-k_{j})$, denoted by
\begin{equation}
\mathcal{A}:=\{(t_{j},x_{j}),(u_{j},y_{j}):u_{j}=1/t_{j},y_{j}=t_{j}^{k_{j}%
}x_{j},y_{j}=t_{j+1},u_{j}=x_{j+1}\}. \label{eq3}%
\end{equation}
Now we prove that $\mathcal{B}_{+}\mathcal{F}$ (respect. $\mathcal{B}%
_{-}\mathcal{F}$) has a multivalued first integral and describe its feature in
this system of coordinates. But first recall that $\mathbb{D}_{r}$ denotes the
disk centered at the origin with radius $r$.

\begin{lemma}
\label{branch integrals}$\mathcal{B}_{+}\mathcal{F}$ (respect. $\mathcal{B}%
_{-}\mathcal{F}$) has a multivalued first integral denoted by $\widetilde
{f}_{+}$ (respect. $\widetilde{f}_{-}$). More precisely, $\widetilde{f}_{+}$
(respect. $\widetilde{f}_{-}$) is given in the system of coordinates
$\mathcal{A}$ by $\widetilde{f}_{+}=\widetilde{f}_{j}$ where
\[
\left\{
\begin{array}
[c]{l}%
\widetilde{f}_{j}(t_{j},x_{j})=t_{j}^{\nu_{j}}x_{j}^{\mu_{j}}U_{j}(t_{j}%
,x_{j})\text{,}\\
\widetilde{f}_{j}(u_{j},y_{j})=u_{j}^{k_{j}\mu_{j}-\nu_{j}}y_{j}^{\mu_{j}%
}V_{j}(u_{j},y_{j})\text{,}%
\end{array}
\right.
\]
with $U_{j},V_{j}\in\mathcal{O}^{\ast}(\mathbb{D}_{\epsilon}\times
\mathbb{D}_{1+\epsilon})$ for some $\epsilon>0$ and all $j=1,\ldots,\ell-1$
(respect. $j=\ell+1,\ldots,n-1$).
\end{lemma}

\begin{proof}
We prove the statement for the positive branch case, the other one being
completely analogous. Pick $\Phi_{\ell+1}\in\operatorname*{Diff}%
_{\widetilde{\mathcal{F}}_{\ell+1},\widetilde{\mathcal{F}}_{\ell+1}^{lin}%
}(\mathcal{H}_{\ell+1},D_{\ell+1})$ and let $\widetilde{f}_{\ell+1}%
:=\Phi_{\ell+1}^{\ast}\widetilde{f}_{\ell+1}^{lin}$. Let $p$ be a regular
point of $D_{\ell+2}$ near the corner $c_{\ell+1,\ell+2}:=D_{\ell+1}\cap
D_{\ell+1}$ and $\Sigma_{p}$ be the fiber of $\mathbb{H}_{\ell+2}$ over $p$.
Recall that $\Phi_{\ell+1}$ induces a bijective map between the spaces of
leaves of $\widetilde{\mathcal{F}}_{\ell+1}$ and $\widetilde{\mathcal{F}%
}_{\ell+1}^{lin}$ which can be realized as $\phi_{\ell+2}\in
\operatorname*{Diff}(\Sigma_{p},p)$. In particular, $\phi_{\ell+2}$ takes
$\operatorname*{Hol}_{\Sigma_{p}}(\widetilde{\mathcal{F}}_{\ell+2},D_{\ell
+2})$ in $\operatorname*{Hol}_{\Sigma_{p}}(\widetilde{\mathcal{F}}_{\ell
+2}^{lin},D_{\ell+2})$. Since $\widetilde{\mathcal{F}}_{\ell+2}$ has just two
singularities, then Proposition \ref{holonomy lifting} ensures that one can
extend $\phi_{\ell+2}$ to $\Phi_{\ell+2}\in\operatorname*{Diff}_{\widetilde
{\mathcal{F}}_{\ell+2},\widetilde{\mathcal{F}}_{\ell+2}^{lin}}(\mathcal{H}%
_{\ell+2},D_{\ell+2})$ by classical path lifting arguments along the fibers of
$\mathbb{H}_{\ell+2}$ (just use the same arguments in the proof of Lemma
\ref{linear models}). Since $\Phi_{\ell+1}$ and $\Phi_{\ell+2}$ induce the
same bijective map between the spaces of leaves of $\widetilde{\mathcal{F}%
}_{\ell+1,\ell+2}$ and $\widetilde{\mathcal{F}}_{\ell+1,\ell+2}^{lin}$, then
$\Phi_{\ell+2}\circ\Phi_{\ell+1}^{-1}$ fixes the leaves of $\widetilde
{\mathcal{F}}_{\ell+1,\ell+2}^{lin}$. Therefore, if we let $\widetilde
{f}_{\ell+2}:=\Phi_{\ell+2}^{\ast}\widetilde{f}_{\ell+2}^{lin}$, then
$\widetilde{f}_{\ell+2}=\widetilde{f}_{\ell+1}$ about $c_{\ell+1,\ell+2}$.
Proceeding by induction on $j>\ell$ we obtain a multivalued first integral for
$\mathcal{B}_{+}\widetilde{\mathcal{F}}$. Finally, let us verify that
$\widetilde{f}_{+}$ has the desired form. Since $\widetilde{f}_{j}^{lin}%
(t_{j},x_{j})=t_{j}^{\nu_{j}}x_{j}^{\mu_{j}}$ and $\Phi_{j}$ is of the form
$\Phi_{j}(t_{j},x_{j})=(t_{j},\alpha_{j}x_{j}+x_{j}a_{j}(t_{j},x_{j}))$, with
$\alpha_{j}\in\mathbb{C}^{\ast}$ and $a_{j}\in\mathfrak{m}_{2}$ (where
$\mathfrak{m}_{2}$ denotes the maximal ideal of $\mathcal{O}_{2}$), then a
straightforward calculation shows that $\widetilde{f}_{j}(t_{j},x_{j}%
)=t_{j}^{\nu_{j}}x_{j}^{\mu_{j}}U_{j}(t_{j},x_{j})$ where $U_{j}(t_{j}%
,x_{j})=[\alpha_{j}+a_{j}(t_{j},x_{j})]^{\mu_{j}}\in\mathcal{O}^{\ast
}(\mathbb{D}_{\epsilon}\times\mathbb{D}_{1+\epsilon})$ for some $\epsilon>0$.
Similarly $\widetilde{f}_{j}^{lin}(u_{j},y_{j})=u_{j}^{k_{j}\mu_{j}-\nu_{j}%
}y_{j}^{\mu_{j}}$ and $\Phi_{j}(u_{j},y_{j})=(u_{j},\beta_{j}y_{j}+y_{j}%
b_{j}(u_{j},y_{j}))$, with $\beta_{j}\in\mathbb{C}^{\ast}$ and $b_{j}%
\in\mathfrak{m}_{2}$. Thus $\widetilde{f}_{j}(u_{j},y_{j})=u_{j}^{k_{j}\mu
_{j}-\nu_{j}}y_{j}^{\mu_{j}}V_{j}(u_{j},y_{j})$ where $V_{j}(u_{j}%
,y_{j})=[\beta_{j}+b_{j}(u_{j},y_{j})]^{\mu_{j}}\in\mathcal{O}^{\ast
}(\mathbb{D}_{\epsilon}\times\mathbb{D}_{1+\epsilon})$ for some $\epsilon>0$.
\end{proof}

\subsubsection{Holomorphic first integrals and the geometry of
$\operatorname*{Sing}(\mathcal{G})$.}

The arguments used in the proof of Lemma \ref{companion fibration} ensure that
$\mathcal{F}$ is resolved together with any \textquotedblleft
generic\textquotedblright\ fiber of the \textit{companion\ fibration}
$\frac{y^{p}}{x^{q}}\equiv\operatorname*{const}$, i.e. $(\mathcal{G}:\eta=0)$
given by $\eta(x,y)=pxdy-qydx$. In other words, $\mathcal{F}$ and
$\mathcal{G}$ are resolved by the same sequence of blowups. In particular, the
minimal resolution of $\mathcal{G}$ has the same tree of projective lines of
the minimal resolution of any element of $\mathcal{QHS}_{\omega,f}^{c,1}$ and
contains its separatrices as fibers. Furthermore, for each $j\neq
\operatorname*{pr}(\widetilde{\mathcal{F}})$ the foliation $\widetilde
{\mathcal{G}}_{j}$ has a (global) holomorphic first integral of the form%
\[
\left\{
\begin{array}
[c]{l}%
\widetilde{\eta}(t_{j},x_{j})=d(t_{j}^{r_{j}}x_{j}^{s_{j}}),\\
\widetilde{\eta}(u_{j},y_{j})=d(u_{j}^{k_{j}s_{j}-r_{j}}y_{j}^{s_{j}}),
\end{array}
\right.
\]
where $r_{j}$, $s_{j}\in\mathbb{N}$ are relatively prime. Since $\widetilde
{\mathcal{G}}_{\operatorname*{pr}(\widetilde{\mathcal{F}})}$ is a radial
fibration, then $\widetilde{\mathcal{G}}_{\operatorname*{pr}(\widetilde
{\mathcal{F}})-1}$ has just one singularity (cf. Figure 8).%

\begin{center}
\includegraphics[
height=0.7074in,
width=4.7469in
]%
{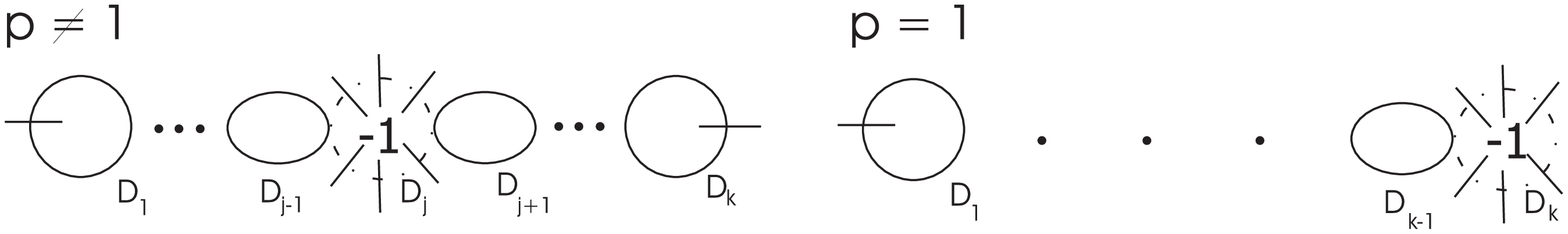}%
\\
Figure 8: The resolution tree of $\mathcal{G}:(\frac{x^{p}}{y^{q}}=const.)$%
\end{center}

\subsubsection{Comparing the indexes of $\mathcal{F}$ and $\mathcal{G}$}

First, recall the celebrated Camacho-Sad's index theorem. Let $S$ be a complex
surface, $C\subset S$ a smooth analytic curve, and $\mathcal{F}$ a germ of a singular foliation defined in a neighborhood of $C$ with just isolated
singularities. For each singular point $p$ of $\mathcal{F}$ in $S$, the
Camacho-Sad's index is defined as follows: choose local coordinates for $S$
around $p$ such that $C$ is given by $(y=0)$. Let $\mathcal{F}$ be given by
$\omega=0$ where $\omega(x,y)=a(x,y)dx+b(x,y)dy$. Then $CS_{p}(\mathcal{F}%
,S)=\operatorname*{Res}_{x=0}\frac{\partial}{\partial y}(\left.  \frac{a}%
{b}(x,y)\right\vert _{y=0})dx$. In particular, if $\omega(x,y)=\mu
y(1+\cdots)dx+\nu x(1+\cdots)dy$ where $\mu,\nu\neq0$, then $CS_{0}%
(\mathcal{F},S)=\frac{\mu}{\nu}$. A\ straightforward calculation shows that
this index does\ not depend on the coordinates.

\begin{theorem}
[Camacho-Sad \cite{CaSa 82a}]Let $S$ be a complex surface, $C\subset M$ a
smooth analytic curve, and $\mathcal{F}$ a germ of a singular foliation defined
in a neighborhood of $S$ with just isolated singularities. Then%
\[
\sum_{p\in\operatorname*{Sing}(\mathcal{F})}CS_{p}(\mathcal{F},S)=C\cdot C
\]
where $C\cdot C$ is the self-intersection number of $C$ in $S$.
\end{theorem}

Now, comparing the Camacho-Sad's indexes of $\widetilde{\mathcal{F}}_{j}$ and
$\widetilde{\mathcal{G}}_{j}$ (starting from $\operatorname*{pr}%
(\widetilde{\mathcal{F}})-1$ to $1$ and from $\operatorname*{pr}%
(\widetilde{\mathcal{F}})+1$ to $n$), then the Camacho-Sad's index theorem
says that
\begin{equation}
\nu_{j}s_{j}-\mu_{j}r_{j}\neq0\text{ for all }j\neq\operatorname*{pr}%
(\widetilde{\mathcal{F}}). \label{eq2}%
\end{equation}

\begin{remark}
If $\operatorname*{Sep}(\mathcal{F})$ is commode, then $\mathcal{F}$ is
automatically generic. In fact, Lemma \ref{companion fibration} ensures that
any Hopf components of $\widetilde{\mathcal{F}}$ about an end of $D$ has just
one singularity. Therefore, with arguments similar to that used for
$\mathcal{G}$, one can verify that each Hopf component $\widetilde
{\mathcal{F}}_{j}$ has linear and periodic holonomy for all $j\neq
\operatorname*{pr}(\widetilde{\mathcal{F}})$. Thus it is linearizable and has
a holomorphic first integral (cf. \cite{MaMo 80}).
\end{remark}

\subsection{Cocycles fixing the leaves of $\mathcal{F}$ and $\mathcal{G}$}

Here we show how to trivialize $\Phi(\mathcal{F})$ and prove Theorem ~B.

\subsubsection{Fixing leaves locally.\label{Fixing locally}}

We introduce some notation first in order to clarify the ideas. Let
$\mathcal{F}$ be a germ of reduced singular foliation in $(\mathbb{C}^{2},0)$.
Since it is characterized by its (local) holonomy group (cf. \cite{MaMo 80},
\cite{MaRa 83}), then it is classical to identify the space of leaves of
$\mathcal{F}$ with the the quotient of $(\mathbb{C}^{2},0)$ by the action of
the unique fibre preserving suspension of this holonomy in
$\operatorname*{Aut}(\mathcal{F},f)$. Therefore, we say that $\phi
\in\operatorname*{Aut}(\mathcal{F})$ fixes the leaves of $\mathcal{F}$ if its
action in the space of leaves of $\mathcal{F}$ is trivial. We denote the set
of such automorphisms by $\operatorname*{Fix}(\mathcal{F})$. As before, this
condition can be verified explicitly by path lifting arguments. In particular,
if $U$ is an open neighborhood of some point in the exceptional divisor of
$\mathcal{B}_{+}\mathcal{F}$ \ (respect. $\mathcal{B}_{-}\mathcal{F}$) and
$\phi\in\operatorname*{Diff}(U)$, then we say that $\phi$ fixes the leaves of
$\mathcal{B}_{+}\mathcal{F}$ (respect. $\mathcal{B}_{-}\mathcal{F}$), denoting
it just by $\phi\in\operatorname*{Fix}(\mathcal{B}_{+}\mathcal{F})$ (respect.
$\mathcal{B}_{-}\mathcal{F}$), if $\phi$ preserves the level sets of the first
integrals introduced in Lemma \ref{branch integrals}.

Let $\mathcal{QHS}_{\omega}$ denote the set of $\mathcal{QHS}$ foliations that
are analytically equivalent to $(\mathcal{F}_{\omega}:\omega=0)$, and $f=0$ be
the separatrix set of $\mathcal{F}_{\omega}$. From the discussion in
\S \ref{anal. cocycles}, in order to determine the moduli space $\left.
\mathcal{QHS}_{\omega,f}^{c}\right/  \mathcal{QHS}_{\omega}$, we have to pick
a fixed model $\mathcal{F}^{o}\in$ $\mathcal{QHS}_{\omega,f}^{c}$ and a
collection of projective charts $(\Phi_{j})$ for any $\mathcal{F}\in$
$\mathcal{QHS}_{\omega,f}^{c}$ (with respect to $\mathcal{F}^{o}$) preserving
$f=0$. In order to simplify the expression of $(\Phi_{j})$, it is natural to
ask it to preserve not just $f=0$ but the whole companion fibration
$\mathcal{G}$. On the other hand, it is not difficult to see that the geometry
of the exceptional divisor of $\mathcal{F}$ allows to simplify inductively the
transversal structure of $\Phi(\mathcal{F})$ in such a way that each
$\Phi_{i,j}$ fixes (locally) the leaves of $\mathcal{F}$. This, of course,
will also simplify the expression of $\Phi(\mathcal{F})$. An optimistic
viewpoint suggests that one can do both at the same time simplifying a lot the
expression of $\Phi(\mathcal{F})$.

\subsubsection{Projective charts and first integrals adapted to a fixed model}

In each componentwise equivalence class pick a model $(\mathcal{F}^{o}%
:\omega^{o}=0)$ and fix first integrals $f_{+}^{o}$ and $f_{-}^{o}$ for
$\mathcal{B}_{+}\mathcal{F}^{o}$ and $\mathcal{B}_{-}\mathcal{F}^{o}$ as in
Lemma \ref{branch integrals}. Now, for any $\mathcal{F}\in\mathcal{QHS}%
_{\omega^{o},f}^{c}$, we shall construct first integrals for $\mathcal{B}%
_{+}\mathcal{F}$ and $\mathcal{B}_{-}\mathcal{F}$ and a collection of
projective charts in an appropriate way. First let us introduce some useful
notation: one says that a collection of projective charts $(\Phi_{j})$ for
$\mathcal{F}\in\mathcal{QHS}_{\omega^{o},f}^{c}$ with respect to
$\mathcal{F}^{o}$ and first integrals $f_{+}$ for $\mathcal{B}_{+}\mathcal{F}$
and $f_{-}$ \ for $\mathcal{B}_{-}\mathcal{F}$ are \textit{adapted} to
$(\mathcal{F}^{o},f_{+}^{o},f_{-}^{o})$ if each $\Phi_{j}$ takes $(f_{+}=c)$
in $(f_{+}^{o}=c)$ for all $j=\ell,\ldots,n$ and $\Phi_{j}$ takes $(f_{-}=c)$
in $(f_{-}^{o}=c)$ for all $j=1,\ldots,\ell$.

\begin{lemma}
\label{adapted}For each $\mathcal{F}\in\mathcal{QHS}_{\omega^{o},f}^{c}$ there
is a collection of projective charts $(\Phi_{j})$ for $\mathcal{F}$ with
respect to $\mathcal{F}^{o}$ and first integrals $f_{+}$ for $\mathcal{B}%
_{+}\mathcal{F}$ and $f_{-}$ \ for $\mathcal{B}_{-}\mathcal{F}$ adapted to
$(\mathcal{F}^{o},f_{+}^{o},f_{-}^{o})$.
\end{lemma}

\begin{proof}
We prove the statement for the positive branch case, the other one being
completely analogous. Pick $\Phi_{\ell}\in\operatorname*{Diff}_{\widetilde
{\mathcal{F}}_{\ell},\widetilde{\mathcal{F}}_{\ell}^{o}}(\mathcal{H}_{\ell
},D_{\ell})$ and let $\widetilde{f}_{\ell,\ell+1}:=\Phi_{\ell}^{\ast
}\widetilde{f}_{\ell,\ell+1}^{o}$, where $\widetilde{f}_{\ell,\ell+1}^{o}$ is
the germ of $\widetilde{f}_{\ell+1}^{o}$ at the corner $c_{\ell,\ell
+1}:=D_{\ell}\cap D_{\ell+1}$. Let $p$ be a regular point of $D_{\ell+1}$ near
the corner $c_{\ell,\ell+1}$ and $\Sigma_{p}$ be the fiber of $\mathbb{H}%
_{\ell+1}$ over $p$. Recall that $\Phi_{\ell}$ induces a bijective map between
the spaces of leaves of $\widetilde{\mathcal{F}}_{\ell,\ell+1}$ and
$\widetilde{\mathcal{F}}_{\ell,\ell+1}^{o}$ which can be realized as
$\phi_{\ell+1}\in\operatorname*{Diff}(\Sigma_{p},p)$. In particular,
$\phi_{\ell+1}$ takes $\operatorname*{Hol}_{\Sigma_{p}}(\widetilde
{\mathcal{F}}_{\ell,\ell+1},D_{\ell+1})$ onto $\operatorname*{Hol}_{\Sigma
_{p}}(\widetilde{\mathcal{F}}_{\ell,\ell+1}^{o},D_{\ell+1})$. Since
$\widetilde{\mathcal{F}}_{\ell+1}$ has just two singularities, then the spaces
of leaves of $\widetilde{\mathcal{F}}_{\ell,\ell+1}$ and $\widetilde
{\mathcal{F}}_{\ell+1}$ coincide as the spaces of leaves of $\widetilde
{\mathcal{F}}_{\ell,\ell+1}^{o}$ and $\widetilde{\mathcal{F}}_{\ell+1}^{o}$.
Therefore Proposition \ref{holonomy lifting} ensures that one can extend
$\phi_{\ell+1}$ to $\Phi_{\ell+1}\in\operatorname*{Diff}_{\widetilde
{\mathcal{F}}_{\ell+1},\widetilde{\mathcal{F}}_{\ell+1}^{o}}(\mathcal{H}%
_{\ell+1},D_{\ell+1})$ along the fibers of $\mathbb{H}_{\ell+1}$ by classical
path lifting arguments (just use the same arguments in the proof of Lemma
\ref{linear models}). Since $\Phi_{\ell}$ and $\Phi_{\ell+1}$ induce the same
bijective map between the spaces of leaves of $\widetilde{\mathcal{F}}%
_{\ell,\ell+1}$ and $\widetilde{\mathcal{F}}_{\ell,\ell+1}^{o}$, then
$\Phi_{\ell}\circ\Phi_{\ell+1}^{-1}$ fixes the leaves of $\widetilde
{\mathcal{F}}_{\ell,\ell+1}^{o}$. Therefore, if we let $\widetilde{f}_{\ell
+1}:=\Phi_{\ell+1}^{\ast}\widetilde{f}_{\ell+1}^{o}$, then $\widetilde
{f}_{\ell+1}=\widetilde{f}_{\ell,\ell+1}$ about $c_{\ell,\ell+1}$. Proceeding
by induction on $j>\ell+1$ we obtain a multivalued first integral for
$\mathcal{B}_{+}\widetilde{\mathcal{F}}$ and the collection of projective
charts with the desired properties.
\end{proof}

In order to give a better understanding of the proof of the next lemma, let us
make a brief digression about the simultaneous linearization of two
transversal non-singular foliations. As it is well known, two germs of
non-singular holomorphic foliations $\mathcal{F}$ and $\mathcal{G}$ can be
simultaneously linearized. In fact the problem can be easily reduced to the
following: given the germs of holomorphic functions $f(x,y)=yU(x,y)$,
$f^{o}(x,y)=y$ and $g(x,y)=x$ about the origin where $U\in\mathcal{O}%
_{2}^{\ast}$, find out $\Phi\in\operatorname*{Diff}(\mathbb{C}^{2},0)$ such
that $\Phi^{\ast}f^{o}=f$ and $\Phi^{\ast}g=g$. If we let $\Phi
(x,y)=(a(x,y),b(x,y))$, then the problem reduces to the following system of
equations%
\[
\left\{
\begin{array}
[c]{l}%
b(x,y)=yU(x,y);\\
a(x,y)=x.
\end{array}
\right.
\]
whose solution is evident. The core of the proof of the following result is
analogous (cf. (\ref{eq4})).

\begin{lemma}
\label{simult. linear.}For each $\mathcal{F}\in\mathcal{QHS}_{\omega^{o}%
,f}^{c}$ and each $j=1,\ldots,\ell-1$ (respect. $j=\ell+1,\ldots,n$) there is
$\Psi_{j}\in\operatorname*{Fix}(\widetilde{\mathcal{G}}_{j})$ such that
$f_{-}^{lin}=\Psi_{j\ast}f_{-}$ (respect. $f_{+}^{lin}=\Psi_{j\ast}f_{+}$).
\end{lemma}

\begin{proof}
We prove the result for the positive branch, the negative one being completely
analogous. In view of the second part of Lemma \ref{branch integrals}, one
just have to find a solution $\Phi_{j}:=\Psi_{j}^{-1}:=(a_{j}(t_{j}%
,x_{j}),b_{j}(t_{j},x_{j}){})$ to the system of equations
\begin{equation}
\left\{
\begin{array}
[c]{l}%
\Phi_{j}^{\ast}\widetilde{f}^{lin}(t_{j},x_{j})=\widetilde{f}^{o}(t_{j}%
,x_{j})\\
\Phi_{j}^{\ast}\widetilde{g}(t_{j},x_{j})=\widetilde{g}(t_{j},x_{j})
\end{array}
\right.  \Leftrightarrow\left\{
\begin{array}
[c]{l}%
a_{j}(t_{j},x_{j})^{\nu_{j}}b_{j}(t_{j},x_{j})_{j}{}^{\mu_{j}}=t_{j}^{\nu_{j}%
}x_{j}^{\mu_{j}}U(t_{j},x_{j})\\
a_{j}(t_{j},x_{j})^{r_{j}}b_{j}(t_{j},x_{j}){}^{s_{j}}=t_{j}^{r_{j}}%
x_{j}^{s_{j}}%
\end{array}
\right.  \label{eq4}%
\end{equation}
But this can be given in the affine charts $(t_{j},x_{j})$ by%
\[
\left\{
\begin{array}
[c]{c}%
a_{j}(t_{j},x_{j})=t_{j}[U(t_{j},x_{j})]^{\frac{s_{j}}{\nu_{j}s_{j}-\mu
_{j}r_{j}}},\\
b_{j}(t_{j},x_{j})=x_{j}[U(t_{j},x_{j})]^{\frac{r_{j}}{\mu_{j}r_{j}-\nu
_{j}s_{j}}}{},
\end{array}
\right.
\]
which is well defined by (\ref{eq2}). A straightforward calculation shows that
the expression of $\Phi_{j}$ in the affine chart $(u_{j},y_{j})$ is given by
\[
\Phi_{j}(u_{j},y_{j})=(u_{j}[V(u_{j},y_{j})]^{\frac{s_{j}}{\mu_{j}r_{j}%
-\nu_{j}s_{j}}},y_{j}[V(u_{j},y_{j})]^{\frac{r_{j}-k_{j}s_{j}}{\mu_{j}%
r_{j}-\nu_{j}s_{j}}})
\]
where $V(u_{j},y_{j}):=U(1/u_{j},u_{j}^{k_{j}}y_{j})\in\mathcal{O}^{\ast
}(\mathbb{D}_{\epsilon},\mathbb{D}_{1+\epsilon})$.
\end{proof}

\begin{remark}
As a straightforward consequence of the above lemma, there is a system of
coordinates $\widetilde{\mathcal{A}}_{j}:=\{(\widetilde{t}_{j},\widetilde
{x}_{j}),(\widetilde{u}_{j},\widetilde{y}_{j})\in\mathbb{C}^{2}:\widetilde
{u}_{j}=1/\widetilde{t}_{j},\widetilde{y}_{j}=\widetilde{t}_{j}^{k_{j}%
}\widetilde{x}_{j}\}$ for $\mathbb{H}_{j}(-k_{j})$ such that the first
integrals of $\widetilde{\mathcal{F}}_{j}$ and $\widetilde{\mathcal{G}}_{j}$
are given respectively by $\widetilde{t}_{j}^{\nu_{j}}x_{j}^{\mu_{j}}$,
$\widetilde{u}_{j}^{k_{j}\mu_{j}-\nu_{j}}\widetilde{y}_{1}^{\mu_{j}}$ and
$\widetilde{t}_{j}^{r_{j}}\widetilde{x}_{j}^{s_{j}}$, $\widetilde{u}%
_{j}^{k_{j}s_{j}-r_{j}}\widetilde{y}_{j}^{s_{j}}$ for all $j\neq\ell$.
\end{remark}

Now we enrich a lit bit the structure preserved by the cocyles.

\begin{lemma}
\label{web}Let $\mathcal{F}\in\mathcal{QHS}_{\omega^{o},f}^{c}$, then there is
a collection of projective charts $(\Phi_{j})$ with respect to $\mathcal{F}%
^{o}$ such that $\Phi_{j}\in\operatorname*{Fix}(\widetilde{\mathcal{G}}_{j})$
for all $j=1,\ldots,n$, $\Phi_{j}\circ\Phi_{j+1}^{-1}\in\operatorname*{Fix}%
(\mathcal{B}_{+}\mathcal{F}^{o})$ for all $j=\ell,\ldots,n-1$ and $\Phi
_{j}\circ\Phi_{j+1}^{-1}\in\operatorname*{Fix}(\mathcal{B}_{-}\mathcal{F}%
^{o})$ for all $j=1,\ldots,\ell-1$.
\end{lemma}

\begin{proof}
From Lemma \ref{adapted} one knows that there is a collection of projective
charts $(\Upsilon_{j})$ for $\mathcal{F}$ with respect to $\mathcal{F}^{o}$
and first integrals $f_{+}$ for $\mathcal{B}_{+}\mathcal{F}$ and $f_{-}$ \ for
$\mathcal{B}_{-}\mathcal{F}$ adapted to $(\mathcal{F}^{o},f_{+}^{o},f_{-}%
^{o})$, where $\Upsilon_{\ell}\in\operatorname*{Fix}(\mathcal{G}_{\ell})$;
thus we let $\Phi_{\ell}:=\Upsilon_{\ell}$. Now we construct $\Phi_{j}$ for
$j\neq\ell$. From Lemma \ref{simult. linear.} there are $\Psi_{j},\Xi_{j}%
\in\operatorname*{Diff}_{\widetilde{\mathcal{F}}_{j},\widetilde{\mathcal{F}%
}_{j}^{lin}}(\mathcal{H}_{j},D_{j})$ such that $\Xi_{j\ast}(f_{+})=f_{+}%
^{lin}$ (respect. $\Xi_{j\ast}(f_{-})=f_{-}^{lin}$), $\Psi_{j\ast}(f_{+}%
^{o})=f_{+}^{lin}$ (respect. $\Psi_{j\ast}(f_{-}^{o})=f_{-}^{lin}$) and
$\Xi_{j},\Psi_{j}\in\operatorname*{Fix}(\mathcal{G}_{j})$. Then define
$\Phi_{j}:=\Psi_{j}^{-1}\circ\Xi_{j}$ in order to obtain the following
commutative diagram%
\begin{equation}%
\begin{array}
[c]{ccc}
& \widetilde{\mathcal{F}}_{j} & \\
^{\Phi_{j}}\swarrow & \circlearrowleft & \searrow^{\Xi_{j}}\\
\widetilde{\mathcal{F}}_{j}^{o} & \overset{\Psi_{j}}{\longrightarrow} &
\widetilde{\mathcal{F}}_{j}^{lin}%
\end{array}
\label{eq5}%
\end{equation}

\end{proof}

\subsubsection{Trivializing cocycles}

Here we follow the program outlined in \S \ref{Fixing locally} in order to
trivialize the cocycles associated with a given fixed model. Recall that
$f_{+}^{o}$ (respect. $f_{-}^{o}$) is the multivalued first integral for
$\mathcal{B}_{+}\mathcal{F}^{o}$ (respect. $\mathcal{B}_{-}\mathcal{F}^{o}$).

\begin{lemma}
\label{linear isotropy extension}Let $\Psi_{j,j+1}\in\operatorname*{Fix}%
(\widetilde{\mathcal{F}}_{j,j+1}^{lin})\cap\operatorname*{Fix}(\widetilde
{\mathcal{G}}_{j,j+1})$ for $j=1,\ldots,n-1$. Then $\Psi_{j,j+1}$ has a unique
extension to $\Psi_{j+1}\in\operatorname*{Fix}(\widetilde{\mathcal{F}}%
_{j+1}^{lin})\cap\operatorname*{Fix}(\widetilde{\mathcal{G}}_{j+1})$ for all
$j\geq\ell$. Analogously, $\Psi_{j,j+1}$ has a unique extension to $\Psi
_{j}\in\operatorname*{Fix}(\widetilde{\mathcal{F}}_{j}^{lin})\cap
\operatorname*{Fix}(\widetilde{\mathcal{G}}_{j})$ for all $j<\ell$.
\end{lemma}

\begin{proof}
We prove the first part of the Lemma, the second one being completely
analogous. We adopt the coordinate system $\mathcal{A}$ introduced in
(\ref{eq3}). Notice that the corner $c_{j,j+1}=D_{j}\cap D_{j+1}$ is
represented by the origin in the affine chart $(t_{j+1},x_{j+1})$ for
$\mathcal{H}_{j+1}$, thus $\Phi_{j,j+1}(t_{j+1},x_{j+1})=(a_{j+1}%
(t_{j+1},x_{j+1}),b_{j+1}(t_{j+1},x_{j+1}))$ where $a_{j+1},b_{j+1}%
\in\mathcal{O}(\mathbb{D}_{\epsilon_{1}}\times\mathbb{D}_{\epsilon_{2}})$.
Since $\Phi_{j,j+1}\in\operatorname*{Fix}(\widetilde{\mathcal{F}}%
_{j,j+1}^{lin})\cap\operatorname*{Fix}(\widetilde{\mathcal{G}}_{j,j+1})$, then
(denoting $i:=j+1$ for simplicity) $a_{i}$ and $b_{i}$ satisfy the following
system of equations
\[
\left\{
\begin{array}
[c]{l}%
a_{i}(t_{i},x_{i})^{\nu_{i}}b_{i}(t_{i},x_{i}){}^{\mu_{i}}=t_{i}^{\nu_{i}%
}x_{i}^{\mu_{i}}\\
a_{i}(t_{i},x_{i})^{r_{i}}b_{i}(t_{i},x_{i}){}^{s_{i}}=t_{i}^{r_{i}}%
x_{i}^{s_{i}}%
\end{array}
\right.
\]
whose solutions are of the form $a_{i}(t_{i},x_{i})=\alpha t_{i}$ and
$b_{i}(t_{i},x_{i})=\beta x{}_{i}$ where\ $\alpha,\frac{1}{\beta}$ are
$(\nu_{i}s_{i}-\mu_{i}r_{i})$-roots of unity. The uniqueness is
straightforward since both $\Phi_{j,j+1}$ and its extension $\Phi_{j+1}$ are holomorphic.
\end{proof}

Now we are in a position to show that the cocycles generated by generic
elements of $\mathcal{QHS}_{\omega^{o},f}^{c}$ are in fact trivial.

\begin{lemma}
\label{isotropy extension}Let $\Phi_{j,j+1}\in\operatorname*{Fix}%
(\widetilde{\mathcal{F}}_{j,j+1}^{o})\cap\operatorname*{Fix}(\widetilde
{\mathcal{G}}_{j,j+1})$ for $j=1,\ldots,n-1$. Then $\Phi_{j,j+1}$ has a unique
extension to $\Phi_{j+1}\in\operatorname*{Fix}(\widetilde{\mathcal{F}}%
_{j+1}^{o})\cap\operatorname*{Fix}(\widetilde{\mathcal{G}}_{j+1})$ for all
$j\geq\ell$. Analogously, $\Phi_{j,j+1}$ has a unique extension to $\Phi
_{j}\in\operatorname*{Fix}(\widetilde{\mathcal{F}}_{j}^{o})\cap
\operatorname*{Fix}(\widetilde{\mathcal{G}}_{j})$ for all $j<\ell$.
\end{lemma}

\begin{proof}
We prove the first part of the Lemma, since the second one is completely
analogous. Let $(\Psi_{j})\in\operatorname*{Fix}(\widetilde{\mathcal{G}}_{j}%
)$, $j=1,\ldots,n$, be the collection of maps introduced in Lemma
\ref{simult. linear.} and $\overline{\Phi}_{j,j+1}:=\Psi_{j+1}\circ
\Phi_{j,j+1}\circ(\Psi_{j+1})^{-1}$. Since $\Psi_{j\ast}f_{+}^{o}=f_{+}^{lin}%
$, then $\overline{\Phi}_{j,j+1}\in\operatorname*{Fix}(\widetilde{\mathcal{F}%
}_{j,j+1}^{lin})\cap\operatorname*{Fix}(\widetilde{\mathcal{G}}_{j,j+1})$ for
all $j=\ell,\ldots,n-1$ (cf. (\ref{eq5})). Hence Lemma
\ref{linear isotropy extension} assures that $\overline{\Phi}_{j,j+1}$ can be
extended to $\overline{\Phi}_{j+1}\in\operatorname*{Fix}(\widetilde
{\mathcal{F}}_{j+1}^{lin})\cap\operatorname*{Fix}(\widetilde{\mathcal{G}%
}_{j+1})$ for all $j=\ell,\ldots,n-1$. Therefore, $\Phi_{j+1}:=(\Psi
_{j+1})^{-1}\circ\overline{\Phi}_{j+1}\circ\Psi_{j+1}\in\operatorname*{Fix}%
(\widetilde{\mathcal{F}}_{j+1}^{o})\cap\operatorname*{Fix}(\widetilde
{\mathcal{G}}_{j+1})$ extends $\Phi_{j,j+1}$. A similar reasoning works for
all $j<\ell$.
\end{proof}

\subsubsection{Extending semi-local conjugations}

Here we use all the machinery developed above in order to prove Theorem ~B. In fact, we show that the vanishing of the cocycles in the
positive (respect. negative) branch means that we can extend to the positive
(respect. negative) branch any conjugation from $\widetilde{\mathcal{F}}%
_{\ell}$ to $\widetilde{\mathcal{F}}_{\ell}^{o}$.

\begin{proof}
[Proof of Theorem ~B]Let $\mathcal{F}^{o}\in\mathcal{QHS}%
_{\omega^{o},f}^{c}$ where $(\mathcal{F}^{o}:\omega^{o}=0)$ is a fixed model.
Let $(\Phi_{j})$ be a collection of projective charts given by Lemma \ref{web}
and $\Phi_{i,j}:=\Phi_{i}\circ\Phi_{j}^{-1}$. Then Lemma
\ref{isotropy extension} ensures that there is $\Xi_{\ell+1}\in
\operatorname*{Fix}(\widetilde{\mathcal{F}}_{\ell+1}^{o})\cap
\operatorname*{Fix}(\widetilde{\mathcal{G}}_{\ell+1})$ such that $\Xi_{\ell
+1}=\Phi_{\ell,\ell+1}$. Let $(\Phi_{j}^{(1)})$ be given by $\Phi_{j}%
^{(1)}:=\Phi_{j}$ for all $j\neq\ell+1$ and $\Phi_{\ell+1}^{(1)}:=\Xi_{\ell
+1}\circ\Phi_{\ell+1}$. Then $(\Phi_{j}^{(1)})$ is a collection of projective
charts such that $\Phi_{j,j+1}^{(1)}\in\operatorname*{Fix}(\widetilde
{\mathcal{F}}_{j,j+1}^{o})\cap\operatorname*{Fix}(\widetilde{\mathcal{G}%
}_{j,j+1})$ and $\Phi_{\ell,\ell+1}^{(1)}=\operatorname*{id}$. Repeating
inductively the same arguments for $j>\ell+1$ we obtain a collection of
projective charts $(\Phi_{j}^{(n-\ell)})$ such that $\Phi_{j,j+1}^{(n-\ell
)}\in\operatorname*{Fix}(\widetilde{\mathcal{F}}_{j,j+1}^{o})\cap
\operatorname*{Fix}(\widetilde{\mathcal{G}}_{j,j+1})$ for all $j=1,\ldots,n-1$
and $\Phi_{j,j+1}^{(n-\ell)}=\operatorname*{id}$ for all $j\geq\ell$. An
analogous reasoning works for all $j<\ell$, generating a collection of
projective charts $(\Phi_{j}^{(n-1)})$ such that $\Phi_{j,j+1}^{(n-1)}%
=\operatorname*{id}$ for all $j=1,\ldots,n-1$. In particular, this family
paste together in order to define a map $\Phi\in\operatorname*{Diff}(M,D)$
such that $\Phi_{\ast}\widetilde{\mathcal{F}}=\widetilde{\mathcal{F}}^{o}$, as desired.
\end{proof}

\end{document}